\documentclass[11pt]{article}
\usepackage{amsmath, array, amstext, amsthm, amssymb, amsfonts, latexsym, amscd, mathtools}
\usepackage{caption}
\usepackage{subcaption}

\newcommand\numberthis{\addtocounter{equation}{1}\tag{\theequation}}
\usepackage{authblk}
\usepackage{tikz}
\usetikzlibrary{shadings,patterns,decorations.pathreplacing}
\newcolumntype{C}{>{$}c<{$}}
\newtheorem{theorem}{Theorem}[section]

\newtheorem{lemma}[theorem]{Lemma}
\newtheorem{proposition}[theorem]{Proposition}
\newtheorem{definition}[theorem]{Definition}
\newtheorem{remark}[theorem]{Remark}
\newtheorem{conjecture}[theorem]{Conjecture}
\newtheorem{example}[theorem]{Example}

\title{\bf On a conjecture of DeLaVi\~na and Waller}
\author[1]{Dinesh Pandey \thanks{Corresponding author: dpandey@wlu.ca}}
\author[2]{Peruvemba Sundaram Ravi \thanks{pravi@wlu.ca}}
\affil[1,2] {Lazaridis School of Business and Economics,
             Wilfrid Laurier University, 75 University Avenue West, Waterloo, Ontario, N2L3C5, Canada.}
\date{}

\begin{document}
\maketitle
\begin{abstract}
The Wiener index of a connected graph is defined as the sum of distances between all its unordered pairs of vertices. Characterising graphs on $n$ vertices with a fixed diameter that maximise the Wiener index is a long-standing open problem. This problem has been resolved fully for trees on $n$ vertices with diameter $d \in \{1,2,3,4,n-3,n-2,n-1\}$ while partial results are available for $d=5$ and $6$. In this context, a conjecture proposed by DeLaVi\~na and Waller has remained open for the last 18 years.

In this paper, we establish a necessary condition for a tree to attain the maximum Wiener index among all trees on $n$ vertices with a given diameter. Using this condition, we characterise the maximal trees for diameter $n-4$ and $n-5$. Furthermore, we prove the DeLaVi\~na Waller conjecture for the classes of graphs having $0,1,2,3$ or $n-4$ cut vertices.\\

\noindent {\bf Keywords:} Wiener index; diameter; cut vertex; extremal problems\\

\noindent {\bf AMS subject classification:}  05C09; 05C12; 05C35;  05C75\\

\noindent {\bf Conflicts of Interest:} The authors declare no conflict of interest.
\end{abstract}

\section{Introduction}
Let $G$ be a finite connected graph with the vertex set $V(G)$ and the edge set $E(G)$. The distance between two vertices $u$ and $v$ of $G$ is defined as the number of edges in a shortest path joining them, and is denoted by $d_G(u,v)$ or by $d(u,v)$ when the graph is clear from the context. The {\it Wiener index} of $G$ is the sum of distances between all its unordered pairs of vertices, denoted by $W(G)$. Thus 
\begin{equation*}
W(G)=\sum_{u,v\in V(G)}{d(u,v)}.
\end{equation*}
The Wiener index was first defined in 1947 by the chemist H.~Wiener~\cite{Wiener}, and has been widely studied in chemistry because many physico-chemical properties of chemical compounds are related to the Wiener index of the graph representing their molecular structure. The paper~\cite{Ng} (1966/67) appears to be the first work in mathematics in which the Wiener index (referred to there as the \emph{sum of all distances}) was studied for digraphs. In 1984 Plesn\'ik \cite{Plesnik} characterised the graphs with the minimum Wiener index among all graphs on $n$ vertices with a fixed diameter (see \cite{Plesnik}, Theorem 2) and raised the problem of identifying the graphs with the maximum Wiener index in the same family. The problem remains completely unsolved except for the trees with diameter $1,2,3,4,5,6, n-3,n-2$ and $n-1$. With some restrictions to this problem DeLaVi\~na and Waller posed the following conjecture in 2008.

\begin{conjecture}[\cite{DeLaVina}, Conjecture 7]\label{the conjecture}
Let $G$ be a graph with diameter $d\geq 3$ and $|V(G)|=2d+1$. Then $W(G)\leq W(C_{2d+1})$.
\end{conjecture}
\noindent  Knor \textit{et al.} have also mentioned this conjecture in their recent survey paper \cite{Knor}. In this paper, we verify the truth of this conjecture for some graph families. In \cite{Pandey} and \cite{Pandey 1}, the authors have studied the extremal graphs with respect to Wiener index in the family of graphs with a fixed number of cut vertices. A major part of this paper is built upon the findings of \cite{Pandey 1}. To show $G_0$ is a graph with maximum Wiener index in a family of graphs $\mathcal{F}$, sometimes it is easy to  find a graph $G'$, not necessarily in $\mathcal{F}$ such that for any $G\in \mathcal{F}$, $W(G)\leq W(G')\leq W(G_0)$. This idea is used to prove our main results in Section \ref{Graphs with cut vertices}.\\

The distance of a vertex $v$ in $G$ is denoted by $D_G(v)$ and defined as 
\begin{equation*}
D_G(v)=\sum_{u\in V(G)}{d(v,u)}.
\end{equation*}
It follows that 
\begin{equation}\label{WI in terms of distance}
W(G)=\frac{1}{2}\sum_{v\in V(G)}{D_G(v)}.
\end{equation}
Prior to the work in \cite{Ng}, extremal problems concerning the distance of a vertex in trees were studied by Harary \cite{Harary}, who referred to this as the status of a vertex. (\ref{WI in terms of distance}) is very useful in counting the Wiener indices of small graphs. Further, the following theorem by Balakrishnan \textit{et al.} \cite{Bala} provides an effective way to compute the Wiener index of a graph with cut vertices. 

\begin{lemma} [\cite{Bala}, Lemma 1.1]\label{count}
Let $G$ be a graph and $w$ be a cut vertex of $G$. Let $G_1$ and $G_2$ be two subgraphs of $G$ such that $G\cong G_1 \cup G_2$ and $V(G_1) \cap V(G_2)= \{w\}$. Then
$$W(G)=W(G_1)+W(G_2)+(|V(G_1)|-1)D_{G_2}(w)+(|V(G_2)|-1)D_{G_1}(w).$$
\end{lemma}

In this paper, every graph $G$ is finite, connected and undirected without loops and parallel edges. $|V(G)|$ denotes the number of vertices in $G$. For any subgraph $H$ of $G$, the {\it order} of $H$ is the number of vertices in $H$.  We denote the {\it diameter} of $G$ by $diam(G)$ and it is  defined as $diam(G)=\max\{d(u,v): u,v\in V(G)\}$. A {\it cut vertex} of $G$ is a vertex whose removal makes the graph disconnected. A  $2$-connected graph is a graph which has no cut vertex. A vertex of degree one is called a {\it pendant vertex}. If there is an edge joining two vertices $u$ and $v$, we say that $u$ and $v$ are adjacent and write $u\sim v$ or $\{uv\}\in E(G)$. A {\it block} in $G$ is a maximal $2$-connected subgraph of $G$. If $G$ is not $2$-connected, then every block of $G$ contains at least one cut vertex of $G$. A block containing exactly one cut vertex of $G$ is called a {\it pendant block} of $G$, otherwise it is called a {\it non-pendant block}. Two blocks are said to be {\it adjacent} if they share a (cut) vertex. The path and the cycle on $n$ vertices are denoted by $P_n$ and $C_n$, respectively. The complete bipartite graph $K_{1,n-1}$ is known as a {\it star} on $n$ vertices. If $G_1$ and $G_2$ are two isomorphic graphs, we write $G_1\cong G_2$. We denote the set of all graphs on $n$ vertices with $k$ cut vertices by $\mathfrak{C}_{n,k}$. The family of graphs on $n$ vertices and $k$ cut vertices with diameter $d$ is denoted by $\mathfrak{C}_{n,k}(d)$, and the family of graphs on $n$ vertices with diameter $d$ is denoted by $\mathfrak{C}_n(d)$. $\mathfrak{T}_{n,k}(d)$ denotes the family of all trees on $n$ vertices and $k$ cut vertices with diameter $d$, and  $\mathfrak{T}_n(d)$ denotes the family of all trees on $n$ vertices with diameter $d$.\\

The paper is organised in the following way. In section \ref{Prelim}, we recall some definitions and results from the literature which will be used in proving our main results. Section \ref{Trees} is a discussion about the trees with the maximum Wiener index in the family of trees on $n$ vertices with a fixed diameter, where we give a necessary condition for a tree to be of maximal Wiener index in this family, and characterise the maximal trees when the diameter is $n-4$ or $n-5$. In Section \ref{edge-minimality}, we discuss edge-minimal graphs which are used as a tool to prove our results in the succeeding section. In Section \ref{Graphs with cut vertices}, we prove the DeLaVi\~na Waller conjecture for the graphs with $0,1,2,3$ or $n-4$ cut vertices. Finally, Section \ref{Conclusions} presents the conclusions and discusses potential avenues for future work.

\section{Preliminaries}\label{Prelim}

The following lemma shows the effect of edge deletion on the Wiener index of a graph, and follows from the definition of the Wiener index.
\begin{lemma}\label{edge deletion}
Let $G$ be a graph and $e\in E(G)$ such that $G-e$ is connected. Then $W(G)<W(G-e)$.
\end{lemma}
\noindent We recall the Wiener indices of some graphs which can be computed easily using (\ref{WI in terms of distance}).
\begin{enumerate} 
\item[$(ii)$] $W(P_n)={n+1 \choose  3}$, 
\item[$(iii)$] $W(K_{1,n-1})=(n-1)^2$,
\item[$(iv)$] $W(C_n)=\begin{cases}  \frac{1}{8}n^3 &\textit{if n is even},\\  \frac{1}{8}n(n^2-1) & \textit{if n is odd.}\end{cases}$
\end{enumerate} 
Also for $u\in V(C_n)$, its distance is
\begin{equation*}
D_{C_n}(u)=\begin{cases}\frac{n^2}{4} &\textit{if n is even}, \\ 
\frac{n^2-1}{4} & \textit{if n is odd.}\end{cases}
\end{equation*}
Let $P_n:v_0v_1\cdots v_{n-1}$ be a path on $n$ vertices. Then it can be readily shown that, for $0\leq i\leq n-1$, 
\begin{align*}
D_{P_n}(v_i)=\frac{n^2+2i^2-(2n-2)i-n}{2}.
\end{align*}

For $g\geq 3$ and $n\geq 4$, let $L_{n,g}$ be the graph obtained by identifying a pendant vertex of the path $P_{n-g+1}$ with a vertex of the cycle $C_g$. $L_{n,g}$ is a unicyclic graph of girth $g$ and belongs to $\mathfrak{C}_{n,n-g}$. It is an important graph, as in many graph families, it (for specific values of $g$) attains maximum Wiener index. For example among unicyclic graphs on $n$ vertices $L_{n,3}$ attains the maximum Wiener index, and among unicyclic graphs with girth $g$, $L_{n,g}$ attains the maximum (see \cite{Feng}). Furthermore, among graphs with at most $k\leq 3$ cut vertices, $L_{n,n-k}$ attains the maximum (see \cite{Pandey 1}). The expression for the Wiener index of $L_{n,g}$ can be found in \cite{Feng} (Theorem 1.1) as the following. 
\begin{equation}\label{WI_Lnk}
W(L_{n,g})=\begin{cases} 
\frac{g^3}{8}+(n-g)(\frac{n^2+ng+3g-1}{6}-\frac{g^2}{12})  &\mbox{ if g is even}, \\ 
\frac{g(g^2-1)}{8}+(n-g)(\frac{n^2+ng+3g-1}{6}-\frac{g^2}{12}-\frac{1}{4})  &\mbox{ if g is odd}.
\end{cases}
\end{equation}


\begin{definition}[\cite{Bapat2}]\label{convexfunction_tree}
Let $T$ be a tree with $|V(T)|\geq 3$ and let $f:V(T)\rightarrow (-\infty,\infty)$ be a function. Then $f$ is said to be {\it convex} if for any distinct $u, v, w\in V(T)$ with $u\sim w\sim v$, $2f(w)\leq f(u)+f(v)$ and strictly convex if $2f(w)< f(u)+f(v).$ We say that $f$ is  quasiconvex if for any distinct $u, v, w\in V(T)$ with $u\sim w\sim v$, $f(w)\leq \max\{f(u), f(v)\}$ and strictly quasiconvex if $f(w)< \max \{f(u),f(v)\}$.
\end{definition}

\begin{proposition} [\cite{Bapat2}, Lemma 2.1]\label{convexity_trees}
Let $T$ be a tree with $|V(T)|\geq 3$ and let $f:V(T)\rightarrow (-\infty,\infty)$ be either strictly convex or strictly quasiconvex. Then $f$ attains its minimum either at a unique vertex or at two adjacent vertices. Furthermore,
\begin{enumerate}
\item[(i)] if $f$ attains its minimum at the unique vertex $v_0$, then for any path $v_0v_1\cdots v_k$ starting at $v_0$, $f(v_0)<f(v_1)<\ldots < f(v_k)$. (We say that $f$ is strictly increasing along any path starting at $v_0$.)
\item[(ii)] If $f$ attains its minimum at the two adjacent vertices $u_0$ and $v_0$, then $f$ is strictly increasing along any path starting at $u_0$ and not containing $v_0$ or starting at $v_0$ and not containing $u_0$. 
\end{enumerate}
\end{proposition}

 \section{Trees with a fixed diameter}\label{Trees}
Trees in $\mathfrak{T}_n(d)$ attaining the maximum Wiener index have been determined for $d \in \{1,2,3,4,n-3,n-2,n-1\}$, while partial results are known for $d=5$ and $6$. 
The problem remains open for $7 \le d \le n-4$. In this section, we establish a necessary condition for a tree to be of maximal Wiener index in $\mathfrak{T}_n(d)$. Using this condition, we determine the maximal trees in $\mathfrak{T}_n(n-4)$ and $\mathfrak{T}_n(n-5)$.\\

 \begin{definition}[\cite{Wagner}, Definition 1]
 Let $(c_1,c_2,\ldots, c_t)$ be a partition of $n-1$. $S(c_1,c_2,\ldots, c_t)$ is the tree assigned to this partition in which the vertices $v_1, v_2, \ldots, v_t$ have degrees $c_1,c_2, \ldots, c_t$, respectively. 
 \end{definition}

\noindent Note that $S(c_1,c_2,\ldots, c_t)$ is a tree on $n$ vertices with diameter at most $4$. Also any tree on $n\geq 5$ vertices and diameter $2\leq d\leq 4$ is of the form $S(c_1,c_2,\ldots, c_t)$ for some partition $(c_1,c_2,\ldots, c_t)$ of $n-1$. Wagner proved the following.
 
 \begin{proposition}[\cite{Wagner}, Theorem 3]\label{Wagner}
 In $\mathfrak{T}_n(d)$, $2\leq d\leq 4$, the tree $S(k,k,\ldots k, k+1,k+1,\ldots k+1)$ where $k=\lfloor\sqrt{n-1}\rfloor$ has the maximal Wiener index. If $k^2+k>n-1$, then $k$ appears $k^2+k-n+1$ times and $k+1$ appears $n-1-k^2$ times. If $k^2+k\leq n-1$, then $k$ appears $k^2+2k-n+2$ times and $k+1$ appears $n-1-k^2-k$ times.
 \end{proposition}
 
 \begin{proposition}[\cite{Mukwembi}, Theorem 3.3]
 Let $T\in \mathfrak{T}_n(5)$. Then $W(T)\leq \frac{9n^2}{4}-2n^{\frac{3}{2}}+O(n)$ and the bound is best possible.
 \end{proposition}
 \begin{proposition}[\cite{Mukwembi}, Theorem 3.4]
 Let $T\in \mathfrak{T}_n(6)$. Then $W(T)\leq 3n^2-2\sqrt{6}n^{\frac{3}{2}} -2n-O(n^{\frac{1}{2}})$ and the bound is best possible.
 \end{proposition} 
 
 A {\it caterpillar} is a tree which becomes a path when all its pendant vertices are deleted. For positive integers $\ell,k$ and $d$ with $n=\ell+k+d$, $T(\ell,k,d)$ denotes the caterpillar obtained from the path $P_d$ by adding $\ell$ pendant vertices adjacent to one pendant vertex of $P_d$ and $k$ pendant vertices adjacent to its other pendant vertex. $T(\ell,k,d)\in \mathfrak{T}_n(d+1)$. 

\begin{proposition}[\cite{Wang}, Theorem 2.10]\label{Caterpillar max}
Among all caterpillars on $n$ vertices and diameter $d$, $T(\lfloor\frac{n-d+1}{2}\rfloor, \lceil\frac{n-d+1}{2}\rceil, d-1)$ uniquely maximizes the Wiener index. Further,
{\scriptsize \begin{equation*}
W\left(T\left(\left\lfloor\frac{n-d+1}{2}\right\rfloor,\left\lceil \frac{n-d+1}{2}\right\rceil,d-1\right)\right)=\begin{cases}{d \choose 3}+\frac{(n-d+1)^2}{4}(d+2)+\frac{n-d+1}{2}[(d-1)^2+d-3] &\mbox{if n-d is odd}, \\ 
																			 {d \choose 3}+\frac{(n-d+1)^2-1}{4}(d+2)+\frac{n-d+1}{2}[(d-1)^2+d-3]+1 &\mbox{if n-d is even}.\end{cases}
\end{equation*}}
\end{proposition}

$P_n$ is the only tree in $\mathfrak{T}_n(n-1)$. Moreover, any tree with diameter $n-2$ is a caterpillar, and therefore by Proposition \ref{Caterpillar max}, the tree with the maximum Wiener index in $\mathfrak{T}_n(n-2)$ is determined. Wang \textit{et al.}~\cite{Wang} further showed that the tree $T(2,2,n-4)$ uniquely maximizes the Wiener index in $\mathfrak{T}_n(n-3)$. Also in \cite {Wang} (Figure 6), the authors have shown some examples of graphs in $\mathfrak{T}_{10}(5)$, $\mathfrak{T}_{11}(5)$ and $\mathfrak{T}_{11}(6)$ which suggest that there can be more than one non-isomorphic trees  attaining the maximum in $\mathfrak{T}_n(d)$, $5\leq d\leq n-4$, which make it difficult to characterise them. We present a subfamily of $\mathfrak{T}_n(d)$ that contains all the potential candidates for the maximal Wiener index in $\mathfrak{T}_n(d)$. 
  
Let $T\in \mathfrak{T}_n(d)$ and $P: v_0v_1\cdots v_d$ be a path of length $d$ in $T$, and let $\bar{T}$ be the forest $T\setminus E(P)$. For $1\leq i\leq d-1$ and $j\geq 0$ define $S_{ij}(T)=\{v\in V(T): d_{\bar{T}}(v_i,v)=j\}$. As $diam(T)=d$, it follows that $j\leq \min\{i,d-i\}$. With this meaning of $S_{ij}(T)$, the generic structure of a tree in $\mathfrak{T}_n(d)$ is presented in Figure \ref{Tree structure}. Note that $S_{i0}(T)=\{v_i\}$ and $S_{ij}(T)$ can be empty for $j\geq 1$. It follows that $S_{ij}(T)=\emptyset\implies S_{i(j+1)}(T)=\emptyset$ for $1\leq i\leq d-1, 1\leq j\leq \min\{i, d-i\}-1$, and $S_{ij}(T)\neq\emptyset\implies S_{i(j-1)}(T)\neq\emptyset$ for $1\leq i\leq d-1, 2\leq j\leq \min\{i, d-i\}$.
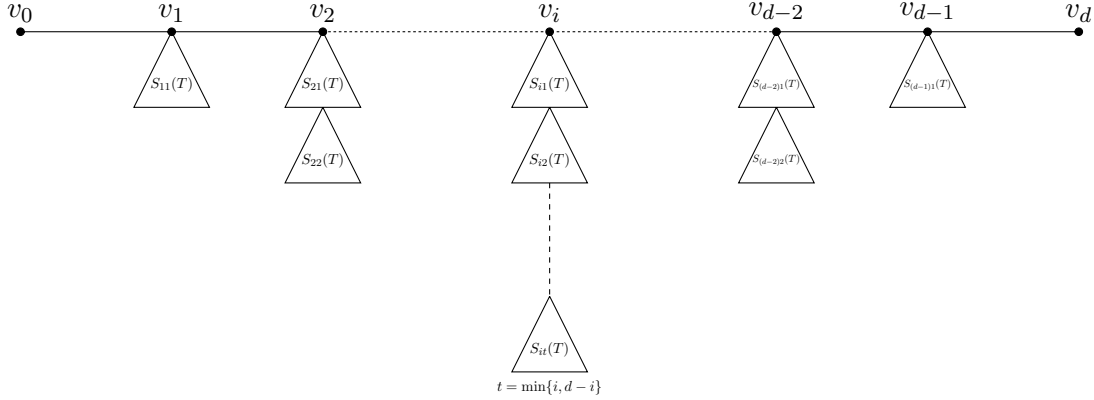
\begin{figure}[h!]
\begin{center}
\begin{tikzpicture}
\filldraw (0,0) node [above]{$v_0$} circle[radius=.5mm]--(2,0)node [above]{$v_1$} circle[radius=.5mm]--(4,0)node [above]{$v_2$} circle[radius=.5mm](7,0)node [above]{$v_i$} circle[radius=.5mm](10,0)node [above]{$v_{d-2}$} circle[radius=.5mm]--(12,0)node [above]{$v_{d-1}$} circle[radius=.5mm]--(14,0)node [above]{$v_d$} circle[radius=.5mm];
\draw (4,0)[dash pattern= on 1pt off 1 pt]--(7,0)--(10,0);
\draw (2,0)--(1.5,-1)--(2.5,-1)--(2,0) (4,0)--(3.5,-1)--(4.5,-1)--(4,0) (4,-1)--(3.5,-2)--(4.5,-2)--(4,-1) (7,0)--(6.5,-1)--(7.5,-1)--(7,0) (7,-1)--(6.5,-2)--(7.5,-2)--(7,-1) (7,-3.5)--(6.5,-4.5)--(7.5,-4.5)--(7,-3.5) (10,0)--(9.5,-1)--(10.5,-1)--(10,0) (10,-1)--(9.5,-2)--(10.5,-2)--(10,-1) (12,0)--(11.5,-1)--(12.5,-1)--(12,0) ;
\draw (7,-2)[dash pattern=on 2pt off 2pt]--(7,-3.5);
\draw (2,-.7) node[scale =.5]{$S_{11}(T)$} (4,-.7) node[scale =.5]{$S_{21}(T)$} (4,-1.7) node[scale =.5]{$S_{22}(T)$} (7,-.7) node[scale =.5]{$S_{i1}(T)$} (7,-1.7) node[scale =.5]{$S_{i2}(T)$} (7,-4.2) node[scale =.5]{$S_{it}(T)$} (10,-.7) node [scale =.36]{$S_{(d-2)1}(T)$}(7,-4.7) node[scale =.5]{$t=\min\{i,d-i\}$} (10,-1.7) node [scale =.36]{$S_{(d-2)2}(T)$} (12,-.7) node [scale =.36]{$S_{(d-1)1}(T)$};
\end{tikzpicture}
\caption{Structure of a tree $T$ of diameter $d$}\label{Tree structure}
\end{center} 
\end{figure}

\begin{lemma}[\cite {Pandey 2}, Theorem 3.6]\label{distance is convex}
Let $T$ be a tree on at least $3$ vertices. Then the distance function $D_T:V(T)\rightarrow (-\infty, \infty)$ is strictly convex.
\end{lemma}

The {\it median} of $T$ is the set of vertices with minimum distance and is denoted by $M(T)$. The following is known about the median of a tree, which also follows from Lemma \ref{distance is convex}.

\begin{lemma}[\cite{Zelinka}, Theorem 2] \label{Median of a tree}
The median of a tree contains either one vertex or two adjacent vertices.
\end{lemma}
A {\it branch} of a tree at a vertex $v$ is a maximal subtree containing $v$ as a pendant vertex. The next theorem gives a necessary condition for a tree to be of maximum Wiener index in $\mathfrak{T}_n(d)$.
\begin{theorem}\label{Tree diameter max}
Let $T\in \mathfrak{T}_n(d)$ such that $W(T_0)=\max\{W(T): T\in \mathfrak{T}_n(d)\}$. Then every vertex of $T_0$ lies on a path of length $d$.
\end{theorem}
\begin{proof}
The result is obvious for $1\leq d\leq 3$. So we assume that $d\geq 4$. Suppose $T_0$ has a vertex $v$ which does not lie on any path of length $d$.  Let $P:v_0v_1\cdots v_d$ be a longest path in $T_0$. Then $v\in S_{gh}(T_0)$ for some $2\leq g\leq d-2$ and $1\leq h< \min\{g,d-g\}$. Without loss of generality, assume that $ \min\{g,d-g\}=g$. Let $P_k$ be a longest path containing $v$. Then one pendant vertex of $P_k$ must belong to $S_{gt}(T_0)$, where $h\leq t< g$. Let $v'$ be the vertex different from $v$, which is nearest to $v$ on the path joining $v$ and $v_g$, that has degree at least $3$. Then $v'\in S_{gh'}(T_0)$ for some $0\leq h'<h$. Let $T_{01}$ be the branch of $T_0$ at $v'$ that contains $v$, and take $T_{02}$ to be the subtree induced by $(T_0\setminus T_{01})\cup \{v'\}$. Then $T_0\cong T_{01}\cup T_{02}$ with $V(T_{01})\cap V(T_{02})=\{v'\}$. Let $\ell$ be the length of a longest path containing $v'$ (and $v$) in $T_{01}$. Cleary $\ell<g$, as $v$ does not lie on any path of length $d$. Also by Lemma \ref{Median of a tree}, $M(T_{02})$ consists of either one vertex or two adjacent vertices. We consider the following two cases.\\

\noindent{\bf Case I:}  All vertices in $M(T_{02})$ lie on a branch at $v'$ containing a vertex from $S_{g(h'+1)}(T_{02})$ adjacent to $v'$.\\

Construct a new graph $T_0'$, from $T_{01}$ and $T_{02}$ by identifying the vertex $v'$ of $T_{01}$ with the vertex $v_{\ell}$ (on $P$) of $T_{02}$ i.e. $T_0'\cong T_{01}\cup T_{02}$ such that $V(T_{01})\cap V(T_{02})=\{v_\ell\}$.  Then $T_0'\in \mathfrak{T}_n(d)$ and $v$ lies on a longest path of $T_0'$. By  Lemma \ref{distance is convex}, the distance function $D_{T_{02}}:V(T_{02})\rightarrow (-\infty, \infty)$ is strictly convex.  As $M(T_{02})$ lies on a branch at $v'$ that contains a vertex from $S_{g(h'+1)}(T_{02})$ adjacent to $v'$, by Proposition \ref {convexity_trees} it follows that $D_{T_{02}}(v_\ell)>D_{T_{02}}(v')$. Also note that $D_{T_{01}}(v_\ell)$ and $D_{T_{01}}(v')$ are the same when considered part of $T_0'$ and $T_0$, respectively. Now by Lemma \ref{count},
\begin{align*}
W(T_0')&= W(T_{01})+W(T_{02})+(|V(T_{02})|-1)D_{T_{01}}(v_\ell)+(|V(T_{01})|-1)D_{T_{02}}(v_\ell)\\
            &>W(T_{01})+W(T_{02})+(|V(T_{02})|-1)D_{T_{01}}(v')+(|V(T_{01})|-1)D_{T_{02}}(v')\\
            &=W(T_0),
\end{align*}
which contradicts the maximality of $T_0$.\\

\noindent {\bf Case II:} At least one vertex in $M(T_{02})$ does not lie on any branch at $v'$ that contains a vertex from $S_{g(h'+1)}(T_{02})$ adjacent to $v'$.\\

 First suppose $v'=v_g$. Let $T_{02}(v',\ell)$ be the branch of $T_{02}$ at $v'$ containing $v_\ell$. If $V(T_{02}(v',\ell))-v'$ does not intersect $M(T_{02})$, construct $T_0'$ by identifying $v'$ of $T_{01}$ with $v_\ell$ of $T_{02}$. If $V(T_{02}(v',\ell))-v'$ intersects $M(T_{02})$, construct $T_0'$ by identifying $v'$ of $T_{01}$ with $v_{d-\ell}$ of $T_{02}$. Then using a similar argument given in Case I, it follows that $W(T_0')> W(T_0)$ which gives a contradiction. \\

Now suppose $v'$ is different from $v_g$.  As $deg(v')\geq 3$ in $T_0$, there exists a vertex $v''$ adjacent to $v'$ in $S_{g(h'+1)}(T_{02})$. Construct a tree $T_0{''}$ from $T_{01}$ and $T_{02}$ by identifying the vertex $v'$ of $T_{01}$ with the vertex $v''$ of $T_{02}$. Then $T_0{''}\in \mathfrak{T}_n(d)$ and $v$ lies on a path of length $k+1$ in $T_0{''}$. From Proposition \ref{convexity_trees}, it follows that $D_{T_{02}}(v'')>D_{T_{02}}(v')$. Now by comparing $W(T_0{''})$ and $W(T_0)$ as in Case I, we get $W(T_0{''})>W(T_0)$, which contradicts the maximality of $T_0$. 
\end{proof}

Theorem \ref{Tree diameter max} helps in deciding whether a tree can be of maximal Wiener index in $\mathfrak{T}_n(d)$ or not. It says that if $T_0$ has maximum Wiener index in $\mathfrak{T}_n(d)$, then $S_{i1}(T_0)\neq \emptyset \implies S_{it}(T_0)\neq \emptyset \;\; \forall i\in \{1,2,\ldots d-1\}$, where $t=\min \{i,d-i\}$.  An example is shown below.

\begin{example}
The trees $T_1$, $T_2$, $T_3$ and $T_4$ of Figure $\ref{Trees in 12-7}$ are members of $\mathfrak{T}_{12}(7)$. $S_{41}(T_1)\neq\emptyset$, but $S_{42}(T_1)=S_{43}(T_1)=\emptyset$. In other words $c_1$ and $d_1$ do not lie on any longest path of $T_1$. Therefore by Theorem \ref{Tree diameter max}, $T_1$ can not have the maximum Wiener index in $\mathfrak{T}_{12}(7)$. Similarly $T_2$ also can not have the maximum Wiener index in $\mathfrak{T}_{12}(7)$ as $S_{51}(T_2)\neq \emptyset$ but $S_{52}(T_2)=\emptyset$. But Theorem  \ref{Tree diameter max} is not helpful in deciding whether $T_3$ or $T_4$ can be of maximal Wiener index in $\mathfrak{T}_{12}(7)$ or not. The Wiener indices of these trees are: $W(T_1)=213, W(T_2)=218, W(T_3)= 236$ and $W(T_4)=230$.  \\
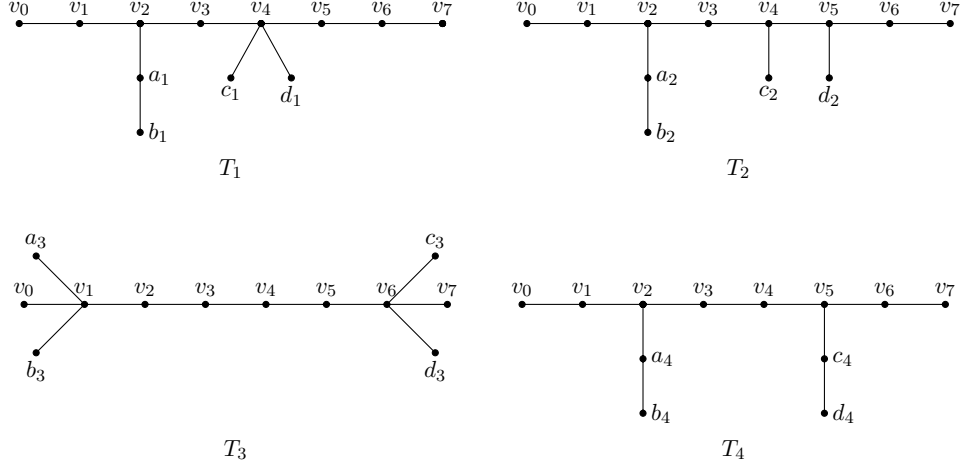
\begin{figure}[h!]
\begin{center}
\begin{tikzpicture}[scale =.8]
\filldraw (0,0)node[scale=.8][above]{$v_0$} circle[radius=.5mm]--(1,0)node[scale=.8][above]{$v_1$} circle[radius=.5mm]--(2,0)node[scale=.8][above]{$v_2$} circle[radius=.5mm]--(3,0)node[scale=.8][above]{$v_3$} circle[radius=.5mm]--(4,0)node[scale=.8][above]{$v_4$} circle[radius=.5mm]--(5,0) node [scale=.8][above]{$v_5$}circle[radius=.5mm]--(6,0)node[scale=.8][above]{$v_6$} circle[radius=.5mm]--(7,0)node[scale=.8][above]{$v_7$}circle [radius=.5mm] circle[radius=.5mm] (2,-1.8) node[scale=.8][right]{$b_1$}circle[radius=.5mm]--(2,-.9) node[scale=.8][right]{$a_1$}circle[radius=.5mm]--(2,0) circle[radius=.5mm] (3.5,-.9) node[scale=.8][below]{$c_1$} circle[radius=.5mm]--(4,0) circle[radius=.5mm]--(4.5,-.9) node[scale=.8][below]{$d_1$} circle[radius=.5mm];
\draw (3.5,-2.4)node[scale=.8]{$T_1$};
\end{tikzpicture}
\hskip .5cm
\begin{tikzpicture}[scale= .8]
\filldraw (0,0)node[scale=.8][above]{$v_0$}  circle[radius=.5mm]--(1,0)node[scale=.8][above]{$v_1$}  circle[radius=.5mm]--(2,0) node[scale=.8][above]{$v_2$} circle[radius=.5mm]--(3,0)node[scale=.8][above]{$v_3$}  circle[radius=.5mm]--(4,0)node[scale=.8][above]{$v_4$}  circle[radius=.5mm]--(5,0)node[scale=.8][above]{$v_5$}  circle[radius=.5mm]--(6,0)node[scale=.8][above]{$v_6$}  circle[radius=.5mm] --(7,0)node[scale=.8][above]{$v_7$} circle [radius=.5mm] (2,-1.8)node[scale=.8][right]{$b_2$}  circle[radius=.5mm]--(2,-.9)node[scale=.8][right]{$a_2$} circle[radius=.5mm]--(2,0) circle[radius=.5mm] (4,-.9)node[scale=.8][below]{$c_2$}  circle[radius=.5mm]--(4,0) (5,0)--(5,-.9)node[scale=.8][below]{$d_2$}  circle[radius=.5mm];
\draw (3.5,-2.4)node[scale=.8]{$T_2$};
\end{tikzpicture}
\vskip .5cm
\begin{tikzpicture}[scale =.8]
\filldraw (0,0)node[scale=.8][above]{$v_0$}  circle[radius=.5mm]--(1,0)node[scale=.8][above]{$v_1$}  circle[radius=.5mm]--(2,0)node[scale=.8][above]{$v_2$}  circle[radius=.5mm]--(3,0)node[scale=.8][above]{$v_3$}  circle[radius=.5mm]--(4,0)node[scale=.8][above]{$v_4$}  circle[radius=.5mm]--(5,0)node[scale=.8][above]{$v_5$}  circle[radius=.5mm]--(6,0)node[scale=.8][above]{$v_6$}  circle[radius=.5mm] --(7,0)node[scale=.8][above]{$v_7$} circle [radius=.5mm] (0.2,-.8)node[scale=.8][below]{$b_3$} circle[radius=.5mm] (.2,.8)node[scale=.8][above]{$a_3$} circle[radius=.5mm](6.8,-.8)node[scale=.8][below]{$d_3$}  circle[radius=.5mm] (6.8,.8)node[scale=.8][above]{$c_3$}  circle[radius=.5mm];
\draw(6.8,-.8) --(6,0)--(6.8,.8) (0.2,-.8)--(1,0)--(.2,.8);
\draw (3.5,-2.4)node[scale=.8]{$T_3$};
\end{tikzpicture}\hskip .5cm
\begin{tikzpicture}[scale =.8]
\filldraw (0,0)node[scale=.8][above]{$v_0$}  circle[radius=.5mm]--(1,0)node[scale=.8][above]{$v_1$}  circle[radius=.5mm]--(2,0)node[scale=.8][above]{$v_2$}  circle[radius=.5mm]--(3,0)node[scale=.8][above]{$v_3$}  circle[radius=.5mm]--(4,0)node[scale=.8][above]{$v_4$}  circle[radius=.5mm]--(5,0)node[scale=.8][above]{$v_5$} circle[radius=.5mm]--(6,0) node[scale=.8][above]{$v_6$} circle[radius=.5mm] --(7,0) node[scale=.8][above]{$v_7$} circle [radius=.5mm] (2,-1.8)node[scale=.8][right]{$b_4$}  circle[radius=.5mm]--(2,-.9)node[scale=.8][right]{$a_4$} circle[radius=.5mm]--(2,0) circle[radius=.5mm] (5,-1.8)node[scale=.8][right]{$d_4$}  circle[radius=.5mm]--(5,-.9)node[scale=.8][right]{$c_4$}  (5,0)--(5,-.9) circle[radius=.5mm];
\draw (3.5,-2.4)node[scale=.8]{$T_4$};
\end{tikzpicture}
\caption {Some trees in $\mathfrak{T}_{12}(7)$}\label{Trees in 12-7}
\end{center}
\end{figure}
\end{example}

Let $T\in \mathfrak{T}_n(d)$ be a tree which contains a vertex that does not belong to any of its longest path. The proof of Theorem \ref{Tree diameter max} describes a  process to construct a new tree $T'\in \mathfrak{T}_n(d)$ from $T$ such that every vertex of $T'$ lies on one of its longest paths and $W(T)<W(T')$. Now suppose we choose two arbitrary trees  $T_0$ and $T_0'$ from $\mathfrak{T}_n(d)$ such that $T_0$ contains a vertex which does not lie on any of its longest paths, and every vertex of $T_0'$ lies on one of its longest paths. A natural question that arises after Theorem \ref{Tree diameter max} is: Is it always true that $W(T_0)< W(T_0')$? The answer is no.  The trees $T_5$ and $T_6$ of Figure \ref{Answer to natural question} give a counter example. $W(T_5)=65>62=W(T_6)$.

\begin{figure}[h!]
\begin{center}
\begin{tikzpicture}[scale=.8]
\filldraw (0,-.5) circle [radius=.5mm]--(1,0) circle [radius=.5mm]--(2,0) circle [radius=.5mm]--(3,0) circle [radius=.5mm]--(4,-.5) circle [radius=.5mm] (0,.5) circle [radius=.5mm]  (4,.5) circle [radius=.5mm]   (2,.8) circle [radius=.5mm]; 
 \draw (0,.5)--(1,0) (3,0)--(4,.5) (2,0)--(2,.8);
\draw (2,-1) node {$T_5$};
\end{tikzpicture}
\hskip 2cm
\begin{tikzpicture}[scale=.8]
\filldraw (1,0) circle [radius=.5mm]--(2,0) circle [radius=.5mm]--(3,0) circle [radius=.5mm]--(4,0) circle [radius=.5mm] (0.2,-.3) circle [radius=.5mm] (0.4,-.8) circle [radius=.5mm] (0.2,.3) circle [radius=.5mm] (0.4,.8) circle [radius=.5mm];
 \draw (.2,-.3)--(1,0)--(0.2,.3) (0.4,-.8)--(1,0)--(0.4,.8);
 \draw (2,-1) node {$T_6$};
\end{tikzpicture}
\end{center}
\caption{Two trees in $\mathfrak{T}_{8}(4)$}\label{Answer to natural question}
\end{figure}
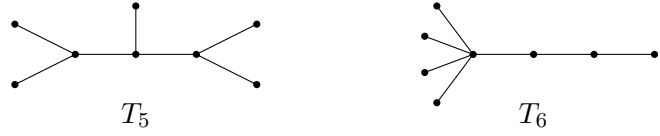

Using Theorem \ref{Tree diameter max}, we now characterise the trees with maximal Wiener indices in $\mathfrak{T}_n(n-4)$. If $n\leq 8$, then $n-4\leq 4$ and the maximal trees with diameter up to $4$ are known. So we consider $n\geq 9$. By Theorem \ref{Tree diameter max}, if $T$ is maximal in  $\mathfrak{T}_n(n-4)$, then all its pendant vertices lie on a path of length $n-4$. So all the potential maximal trees (up to isomorphism) can be drawn in the form of Figure \ref{Tree structure}. If a maximal tree is a caterpillar, then by Proposition \ref{Caterpillar max}, it must be $T(2,3,n-5)$. Hence, for $n\geq 10$ the potential maximal trees are precisely $T_7, T_8,T_9$ and $T_{10}$ in Figure \ref{maximal n,n-4}, and for $n=9$, they are precisely $T_7, T_8$ and $T_9$.
\begin{figure}[h!]
\begin{center}
\begin{tikzpicture}
\filldraw (0,0)node[above]{$v_0$} circle [radius=.5mm]--(1,0)node[above]{$v_1$}circle [radius=.5mm]--(2,0)node[above]{$v_2$}circle [radius=.5mm](4,0)node[above]{$v_{n-5}$}circle [radius=.5mm]--(5,0)node[above]{$v_{n-4}$}circle [radius=.5mm] (1,-.8)circle [radius=.5mm] (4.2,-.8)circle [radius=.5mm] (3.8,-.8)circle [radius=.5mm];
\draw (1,0)--(1,-.8) (3.8,-.8)--(4,0)--(4.2,-.8);
\draw(2,0)[dash pattern= on 1pt off 1pt]--(4,0) (2.5,-2) node {$T_7$};
\end{tikzpicture}
\hskip 1cm 
\begin{tikzpicture}
\filldraw (0,0)node[above]{$v_0$} circle [radius=.5mm]--(1,0)node[above]{$v_1$}circle [radius=.5mm]--(2,0)node[above]{$v_2$}circle [radius=.5mm](4,0)node[above]{$v_{n-5}$}circle [radius=.5mm]--(5,0)node[above]{$v_{n-4}$}circle [radius=.5mm] (1,-.8)circle [radius=.5mm] (2,-.8)circle [radius=.5mm] (2,-1.6)circle [radius=.5mm];
\draw (1,0)--(1,-.8) (2,0)--(2,-.8)--(2,-1.6);
\draw(2,0)[dash pattern= on 1pt off 1pt]--(4,0) (2.5,-2) node{$T_8$};
\end{tikzpicture}
\vskip .4cm 
\begin{tikzpicture}
\filldraw (0,0)node[above]{$v_0$} circle [radius=.5mm]--(1,0)node[above]{$v_1$}circle [radius=.5mm]--(2,0)node[above]{$v_2$}circle [radius=.5mm](4,0)node[above]{$v_{n-5}$}circle [radius=.5mm]--(5,0)node[above]{$v_{n-4}$}circle [radius=.5mm] (4,-.8)circle [radius=.5mm] (2,-.8)circle [radius=.5mm] (2,-1.6)circle [radius=.5mm];
\draw (4,0)--(4,-.8) (2,0)--(2,-.8)--(2,-1.6);
\draw(2,0)[dash pattern= on 1pt off 1pt]--(4,0) (2.5,-2.2) node{$T_9$};
\end{tikzpicture}
\hskip 1cm 
\begin{tikzpicture}
\filldraw (0,0)node[above]{$v_0$} circle [radius=.5mm]--(1,0)node[above]{$v_1$}circle [radius=.5mm]--(2,0)node[above]{$v_2$}circle [radius=.5mm]--(3,0)node[above]{$v_3$}circle [radius=.5mm] (5,0)node[above]{$v_{n-5}$}circle [radius=.5mm]--(6,0)node[above]{$v_{n-4}$}circle [radius=.5mm] (3,-.6)circle [radius=.5mm] (3,-1.2)circle [radius=.5mm] (3,-1.8)circle [radius=.5mm];
\draw  (3,0)--(3,-.6)--(3,-1.2)--(3,-1.8);
\draw(3,0)[dash pattern= on 1pt off 1pt]--(5,0) (2.5,-2.2) node{$T_{10}$};
\end{tikzpicture}
\caption {The potential candidates for maximal Wiener index in $\mathfrak{T}_{n}(n-4)$}\label{maximal n,n-4}
\end{center}
\end{figure}
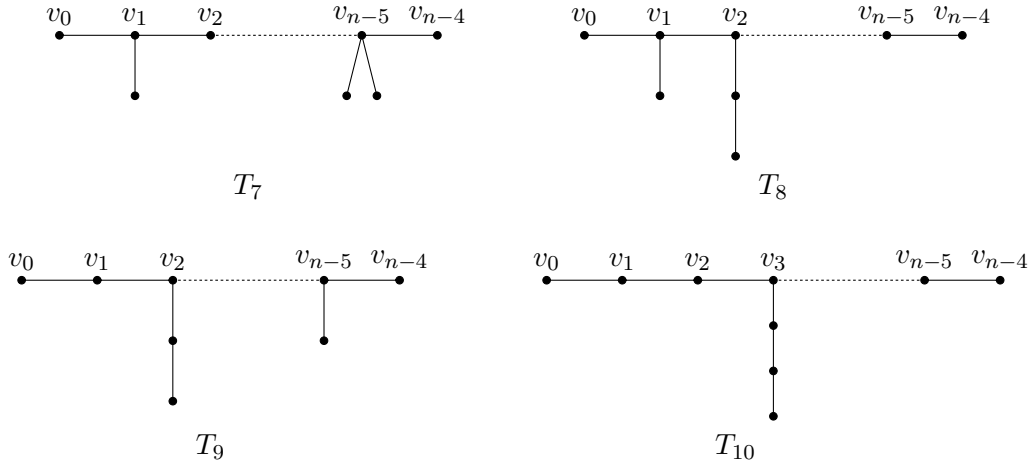 
We calculate (using Lemma \ref{count}) the Wiener indices of the trees in Figure \ref{maximal n,n-4} as the following: $W(T_7)=\frac{n^3-25n+84}{6}$, $W(T_8)=\frac{n^3-43n+234}{6}$, $W(T_9)=\frac{n^3-31n+138}{6}$ and $W(T_{10})=\frac{n^3-55n+378}{6}$. Now the Wiener indices of these trees can easily be compared. If $n\geq 10$ we get $W(T_7)> W(T_i)$ for $i=8,9,10$. For $n=9$, $W(T_7)=W(T_9)=98$ and $W(T_8)=96$. Thus we conclude the following.
\begin{itemize}
\item Both the trees $T_7$ and $T_9$ maximize the Wiener index in $\mathfrak{T}_9(5)$.
\item If $n\geq 10$, $T_7$ uniquely maximizes the Wiener index in $\mathfrak{T}_n(n-4)$. 
\end{itemize}

Using the same idea the maximal trees in $\mathfrak{T}_n(n-5)$ can also be characterised. By Theorem \ref{Tree diameter max}, the trees $T_i$, $11\leq i\leq 20$ of Figure \ref{maximal n,n-5} are the only possibilities for trees with maximum Wiener index in $\mathfrak{T}_n(n-5)$. Here we take $n\geq 10$ as we are looking for trees with diameter at least $5$. We calulate the Wiener indices of $T_i$, $11\leq i\leq 20$ as the following: $W(T_{11})=\frac{n^3-37n+132}{6}$, $W(T_{12})=\frac{n^3-43n+186}{6}$, $W(T_{13})=\frac{n^3-49n+252}{6}$, $W(T_{14})=\frac{n^3-67n+414}{6}$, $W(T_{15})=\frac{n^3-67n+402}{6}$, $W(T_{16})=\frac{n^3-73n+456}{6}$, $W(T_{17})=\frac{n^3-49n+240}{6}$,  $W(T_{18})=\frac{n^3-61n+396}{6}$, $W(T_{19})=\frac{n^3-79n+594}{6}$, $W(T_{20})=\frac{n^3-97n+864}{6}$. Note that $T_{20}$ does not  exists for $n\in \{10,11,12\}$. Further $T_{18}$ and $T_{19}$ do not exist for $n=10$. For $n\geq 13$, each $T_i$, $11\leq i\leq 20$  exists. Now the Wiener indices of $T_i$ for $11\leq i\leq 20$ can be easily compared. If $n=10$, then $W(T_{11})\geq W(T_i)$ for $12\leq i \leq 17$ and the equality holds iff $i=13$. If $n=11$, $W(T_{11})\geq W(T_i)$ for $12\leq i \leq 19$ and the equality holds iff $i\in\{18,19\}$ ($T_{18}\cong T_{19}$ for $n=11$). If $n\geq 12$, then $W(T_{11})>W(T_i)$ for $12\leq i\leq 20$, provided they exist. So we conclude the following.
\begin{itemize}
\item Both the trees $T_{11}$ and $T_{13}$ maximize the Wiener index in $\mathfrak{T}_{10}(5)$.
\item Both the trees $T_{11}$ and $T_{18}$ maximize the Wiener index in $\mathfrak{T}_{11}(6)$.
\item For $n\geq 12$, $T_{11}$ uniquely maximizes the Wiener index in $\mathfrak{T}_n(n-5)$.
\end{itemize}

\begin{figure}[h!]
\begin{center}
\begin{tikzpicture}[scale=.8]
\filldraw (0,0)node[scale=.8][above]{$v_0$} circle [radius=.5mm]--(1,0)node[scale=.8][above]{$v_1$}circle [radius=.5mm]--(2,0)circle [radius=.5mm](4,0)node[scale=.8][above]{$v_{n-6}$}circle [radius=.5mm]--(5,0)node[scale=.8][above]{$v_{n-5}$}circle [radius=.5mm] (.8,-.8)circle [radius=.5mm] (1.2,-.8)circle [radius=.5mm] (4.2,-.8)circle [radius=.5mm] (3.8,-.8)circle [radius=.5mm];
\draw (.8,-.8)--(1,0)--(1.2,-.8) (3.8,-.8)--(4,0)--(4.2,-.8);
\draw(2,0)[dash pattern= on 1pt off 1pt]--(4,0) (2.5,-2) node {$T_{11}$};
\end{tikzpicture}
\hskip .5cm 
\begin{tikzpicture}[scale=.8]
\filldraw (0,0)node[scale=.8][above]{$v_0$} circle [radius=.5mm]--(1,0)node[scale=.8][above]{$v_1$}circle [radius=.5mm]--(2,0)node[scale=.8][above]{$v_2$}circle [radius=.5mm](4,0)node[scale=.8][above]{$v_{n-6}$}circle [radius=.5mm]--(5,0)node[scale=.8][above]{$v_{n-5}$}circle [radius=.5mm] (2,-.8)circle [radius=.5mm] (2,-1.6)circle [radius=.5mm] (4.2,-.8)circle [radius=.5mm] (3.8,-.8)circle [radius=.5mm];
\draw (2,0)--(2,-.8)--(2,-1.6) (3.8,-.8)--(4,0)--(4.2,-.8);
\draw(2,0)[dash pattern= on 1pt off 1pt]--(4,0) (2.5,-2) node {$T_{12}$};
\end{tikzpicture}
\hskip .5cm 
\begin{tikzpicture}[scale=.8]
\filldraw (0,0)node[scale=.8][above]{$v_0$} circle [radius=.5mm]--(1,0)node[scale=.8][above]{$v_1$}circle [radius=.5mm]--(2,0)node[scale=.8][above]{$v_2$}circle [radius=.5mm](4,0)node[scale=.8][above]{$v_{n-6}$}circle [radius=.5mm]--(5,0)node[scale=.8][above]{$v_{n-5}$}circle [radius=.5mm] (1,-.8)circle[radius=.5mm] (2,-.8)circle [radius=.5mm] (2,-1.6)circle [radius=.5mm] (4,-.8)circle [radius=.5mm];
\draw (1,0)--(1,-.8) (2,0)--(2,-.8)--(2,-1.6) (4,0)--(4,-.8);
\draw(2,0)[dash pattern= on 1pt off 1pt]--(4,0) (2.5,-2) node {$T_{13}$};
\end{tikzpicture}
\hskip .5cm 
\begin{tikzpicture}[scale=.8]
\filldraw (0,0)node[scale=.8][above]{$v_0$} circle [radius=.5mm]--(1,0)node[scale=.8][above]{$v_1$}circle [radius=.5mm]--(2,0)node[scale=.8][above]{$v_2$}circle [radius=.5mm](4,0)node[scale=.8][above]{$v_{n-6}$}circle [radius=.5mm]--(5,0)node[scale=.8][above]{$v_{n-5}$}circle [radius=.5mm] (1,-.8)circle[radius=.5mm] (2,-.8)circle [radius=.5mm] (2.2,-1.6)circle [radius=.5mm] (1.8,-1.6)circle [radius=.5mm];
\draw (1,0)--(1,-.8) (2,0)--(2,-.8) (1.8,-1.6)--(2,-.8)--(2.2,-1.6);
\draw(2,0)[dash pattern= on 1pt off 1pt]--(4,0) (2.5,-2) node {$T_{14}$};
\end{tikzpicture}
\hskip .5cm 
\begin{tikzpicture}[scale=.8]
\filldraw (0,0)node[scale=.8][above]{$v_0$} circle [radius=.5mm]--(1,0)node[scale=.8][above]{$v_1$}circle [radius=.5mm]--(2,0)node[scale=.8][above]{$v_2$}circle [radius=.5mm](4,0)node[scale=.8][above]{$v_{n-6}$}circle [radius=.5mm]--(5,0)node[scale=.8][above]{$v_{n-5}$}circle [radius=.5mm] (2,-.8)circle [radius=.5mm] (2,-1.6)circle[radius=.5mm] (2.3,-1.6)circle [radius=.5mm] (1.7,-1.6)circle [radius=.5mm];
\draw  (2,0)--(2,-.8)--(2,-1.6) (1.7,-1.6)--(2,-.8)--(2.3,-1.6);
\draw(2,0)[dash pattern= on 1pt off 1pt]--(4,0) (2.5,-2) node {$T_{15}$};
\end{tikzpicture}
\hskip .5cm 
\begin{tikzpicture}[scale=.8]
\filldraw (0,0)node[scale=.8][above]{$v_0$} circle [radius=.5mm]--(1,0)node[scale=.8][above]{$v_1$}circle [radius=.5mm]--(2,0)node[scale=.8][above]{$v_2$}circle [radius=.5mm](4,0)node[scale=.8][above]{$v_{n-6}$}circle [radius=.5mm]--(5,0)node[scale=.8][above]{$v_{n-5}$}circle [radius=.5mm] (1.6,-.7)circle[radius=.5mm] (1.2,-1.4)circle [radius=.5mm]   (2.4,-.7)circle [radius=.5mm] (2.8,-1.4)circle [radius=.5mm];
\draw (1.2,-1.4)--(1.6,-.7)--(2,0)--(2.4,-.7)--(2.8,-1.4);
\draw(2,0)[dash pattern= on 1pt off 1pt]--(4,0) (2.5,-2) node {$T_{16}$};
\end{tikzpicture}
\hskip .5cm 
\begin{tikzpicture}[scale=.8]
\filldraw (0,0)node[scale=.8][above]{$v_0$} circle [radius=.5mm]--(1,0)node[scale=.8][above]{$v_1$}circle [radius=.5mm]--(2,0)node[scale=.8][above]{$v_2$}circle [radius=.5mm](4,0)node[scale=.8][above]{$v_{n-7}$}circle [radius=.5mm]--(5,0)node[scale=.8][above]{$v_{n-6}$}circle [radius=.5mm]--(6,0)node[scale=.8][above]{$v_{n-5}$}circle [radius=.5mm] (2,-.8)circle[radius=.5mm] (2,-1.6)circle [radius=.5mm]   (4,-.8)circle [radius=.5mm] (4,-1.6)circle [radius=.5mm];
\draw (2,0)--(2,-.8)--(2,-1.6) (4,0)--(4,-.8)--(4,-1.6);
\draw(2,0)[dash pattern= on 1pt off 1pt]--(4,0) (3,-3) node {$T_{17}$};
\end{tikzpicture}
\hskip 1cm 
\begin{tikzpicture}[scale=.8]
\filldraw (0,0)node[scale=.8][above]{$v_0$} circle [radius=.5mm]--(1,0)node[scale=.8][above]{$v_1$}circle [radius=.5mm]--(2,0)node[scale=.8][above]{$v_2$}circle [radius=.5mm]--(3,0)node[scale=.8][above]{$v_3$}circle [radius=.5mm](6,0)node[scale=.8][above]{$v_{n-6}$}circle [radius=.5mm]--(7,0)node[scale=.8][above]{$v_{n-5}$}circle [radius=.5mm] (3,0)--(3,-.8)circle[radius=.5mm]--(3,-1.6)circle [radius=.5mm]-- (3,-2.4)circle [radius=.5mm] (6,-.8)circle [radius=.5mm]--(6,0);
\draw(3,0)[dash pattern= on 1pt off 1pt]--(6,0) (3.5,-3) node {$T_{18}$};
\end{tikzpicture}
\hskip 1cm 
\begin{tikzpicture}[scale=.8]
\filldraw (0,0)node[scale=.8][above]{$v_0$} circle [radius=.5mm]--(1,0)node[scale=.8][above]{$v_1$}circle [radius=.5mm]--(2,0)node[scale=.8][above]{$v_2$}circle [radius=.5mm]--(3,0)node[scale=.8][above]{$v_3$}circle [radius=.5mm](6,0)node[scale=.8][above]{$v_{n-6}$}circle [radius=.5mm]--(7,0)node[scale=.8][above]{$v_{n-5}$}circle [radius=.5mm] (3,0)--(3,-.8)circle[radius=.5mm]--(3,-1.6)circle [radius=.5mm]-- (2.6,-2.4)circle [radius=.5mm] (3.4,-2.4)circle [radius=.5mm]--(3,-1.6);
\draw(3,0)[dash pattern= on 1pt off 1pt]--(6,0) (3.5,-4) node {$T_{19}$};
\end{tikzpicture}
\hskip 1cm 
\begin{tikzpicture}[scale=.8]
\filldraw (0,0)node[scale=.8][above]{$v_0$} circle [radius=.5mm]--(1,0)node[scale=.8][above]{$v_1$}circle [radius=.5mm]--(2,0)node[scale=.8][above]{$v_2$}circle [radius=.5mm]--(3,0)node[scale=.8][above]{$v_3$}circle [radius=.5mm]--(4,0)node[scale=.8][above]{$v_4$}circle [radius=.5mm] (7,0)node[scale=.8][above]{$v_{n-6}$}circle [radius=.5mm]--(8,0)node[scale=.8][above]{$v_{n-5}$}circle [radius=.5mm] (4,0)--(4,-.8)circle[radius=.5mm]--(4,-1.6)circle [radius=.5mm]-- (4,-2.4)circle [radius=.5mm]--(4,-3.2)circle [radius=.5mm];
\draw(4,0)[dash pattern= on 1pt off 1pt]--(7,0) (4.5,-4) node {$T_{20}$};
\end{tikzpicture}
\caption {The potential candidates for maximal Wiener index in $\mathfrak{T}_{n}(n-5)$}\label{maximal n,n-5}
\end{center}
\end{figure}
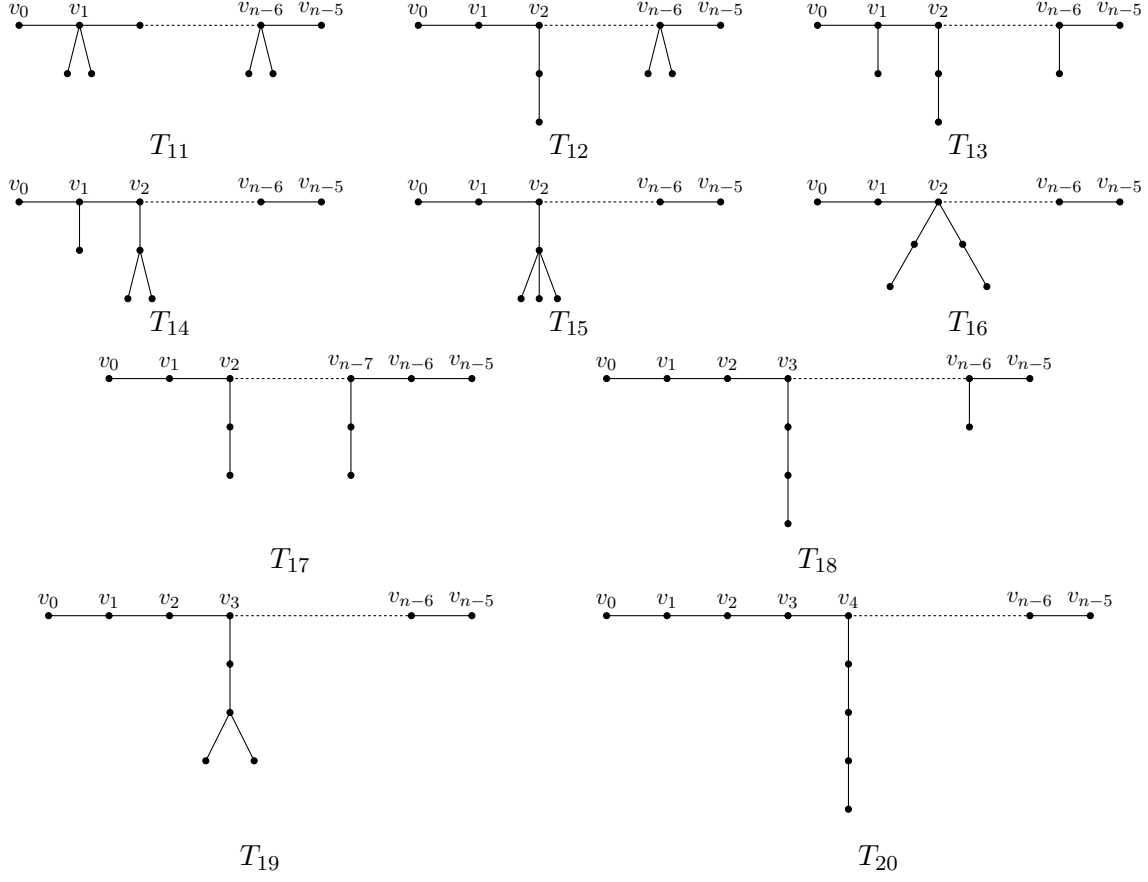 

 The eccentricity $e(v)$ of a vertex $v$ in $T$ is defined as $e(v)=\max\{d(v,u): u\in V(T)\}$. The {\it center} of $T$ is the set of vertices with minimum eccentricity. We denote the center of $T$ by $C(T)$. Let $G$ be a graph and $S_1, S_2$ be two subsets of $V(G)$. The distance between $S_1$ and $S_2$ is defined as $d(S_1,S_2)=\min \{d(v_1,v_2): v_1\in S_1, v_2\in S_2\}.$ As the value of $k$ increases, the list of trees in $\mathfrak{T}_n(n-k)$ that may attain the maximum Wiener index becomes larger, which makes it difficult to describe the maximal trees through manual analysis. However, the maximal trees can still be identified for a few more values of $k$, such as $6,7$ and $8$. By imposing further structural constraints on the possible maximal trees, the list of potential candidates can further be reduced. In particular, we propose that an additional restriction can be formulated in terms of the distance between the centre and the median of the trees. We therefore conclude this section with the following conjecture, which we believe to be true.

\begin{conjecture}
Let $W(T_0)=\max\{W(T): T\in \mathfrak{T}_n(d)\}$. Then $d(C(T_0), M(T_0))=0$.
\end{conjecture}


\section{Edge-minimal graphs in $\mathfrak{C}_{n,k}(d)$ and $\mathfrak{C}_n(d)$}\label{edge-minimality}
 An {\it edge-minimal} graph in $\mathfrak{C}_{n,k}(d)$ is a graph such that $G-e\not\in \mathfrak{C}_{n,k}(d)$ for any $e\in E(G)$. It is clear that if a graph has maximal Wiener index over $\mathfrak{C}_{n,k}(d)$, then it is edge-minimal. Let $G$ be a graph, and let $G'$ be obtained by deleting some of its edges. Then we say $G$ is reducible to $G'$ by edge deletion or simply $G$ is reducible to $G'$.  If $G'$ can not be obtained by deleting edges of $G$, we say $G$ is not reducible to $G'$  (or $G$ can not be reduced to $G'$). If $G$ is reducible to $G'$, then by Lemma \ref{edge deletion}, it follows that $W(G)<W(G')$. A graph in $\mathfrak{C}_{n,k}(d)$ is either edge-minimal or is reducible to an edge-minimal graph in $\mathfrak{C}_{n,k}(d)$. For example, in Figure \ref{reduce edge-minimal}, the Graph $G_0$ is not edge-minimal in $\mathfrak{C}_{7,2}(3)$, but by deleting the edges $e_1$ and $e_2$, we obtain the graph $G_2$, which is edge-minimal in $\mathfrak{C}_{7,2}(3)$.

\begin{figure}[h!]
\begin{center}
\begin{tikzpicture}[scale =.8]
\filldraw (0,0) circle [radius=.5mm]--(1.2,0) circle [radius=.5mm]--(1.5,1.2) circle [radius=.5mm]--(.6,2) circle [radius=.5mm]--(-0.2,1.2) circle [radius=.5mm]--(0,0) (1.7,2.7) circle [radius=.5mm] (-.4,2.7)circle [radius=.5mm];
\draw (-.4,2.7)--(-.2,1.2) (1.5,1.2)--(1.7,2.7) (-.2,1.2)--(1.4,1.2);
\draw(.2,1.6) node [scale=.8][above]{$e_1$} (.6,0) node [scale=.8][above]{$e_2$} (.6,-1) node {$G_0$};
\end{tikzpicture}
\hskip 1 cm
\begin{tikzpicture}[scale =.8]
\filldraw (0,0) circle [radius=.5mm]--(1.2,0) circle [radius=.5mm]--(1.5,1.2) circle [radius=.5mm]--(.6,2) circle [radius=.5mm](-0.2,1.2) circle [radius=.5mm]--(0,0) (1.7,2.7) circle [radius=.5mm] (-.4,2.7)circle [radius=.5mm];
\draw (-.4,2.7)--(-.2,1.2) (1.5,1.2)--(1.7,2.7) (-.2,1.2)--(1.4,1.2);
\draw (.6,0) node [scale=.8][above]{$e_2$} (.6,-1) node {$G_1\cong G_0-e_1$};
\end{tikzpicture}
\hskip 1 cm
\begin{tikzpicture}[scale =.8]
\filldraw (0,0) circle [radius=.5mm](1.2,0) circle [radius=.5mm]--(1.5,1.2) circle [radius=.5mm]--(.6,2) circle [radius=.5mm](-0.2,1.2) circle [radius=.5mm]--(0,0) (1.7,2.7) circle [radius=.5mm] (-.4,2.7)circle [radius=.5mm];
\draw (-.4,2.7)--(-.2,1.2) (1.5,1.2)--(1.7,2.7) (-.2,1.2)--(1.4,1.2);
\draw (.6,-1) node {$G_2\cong G_1-e_2$};
\end{tikzpicture}
\end{center}
\caption{Reducing a graph in $\mathfrak{C}_{7,2}(3)$ to an edge-minimal graph in $\mathfrak{C}_{7,2}(3)$}\label{reduce edge-minimal}
\end{figure}
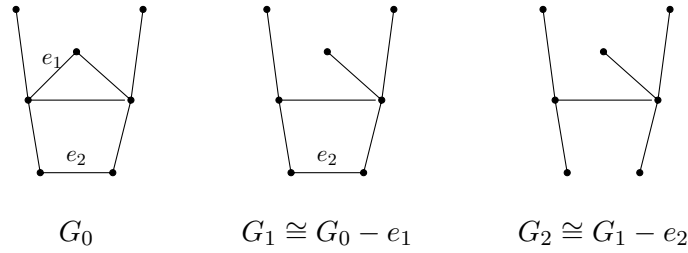

\noindent Further, any graph can be reduced to a tree by edge deletion. But a graph in $\mathfrak{C}_{n,k}(d)$ may not be reducible to a tree in $\mathfrak{T}_{n,k}(d)$ or in $\mathfrak{T}_n(d)$. For example each of $G_4$, $G_5$ and $G_6$ of  Figure \ref{edge-minimal in C-9-3-4-triangle}  belongs to $\mathfrak{C}_{9,3}(4)$ and is reducible to a tree in $\mathfrak{T}_9(4)$, whereas the graph $G_3$ of Figure \ref {edge-minimal in C-7-3-3} belongs to $\mathfrak{C}_{7,3}(3)$ but is not reducible to a tree in $\mathfrak{T}_7(3)$.\\

We now identify several edge-minimal graphs within certain graph families, which will play a key role in the proof of our main result in Section \ref{Graphs with cut vertices}. Table \ref{Trees in 7-2-3} is a list of all the non isomorphic edge-minimal graphs in $\mathfrak{C}_{7,2}(3)$. There is only one edge-minimal graph in $\mathfrak{C}_{7,3}(3)$, which is shown in Figure \ref{edge-minimal in C-7-3-3}. We now turn to the edge-minimal graphs in $\mathfrak{C}_{9,3}(4)$ that are not reducible to a tree in $\mathfrak{T}_9(4)$. The following lemma provides a characterisation of such graphs.
\begin{table}[h!]
\begin{center}
\begin{tabular}{|c|c|}
\hline
$G$ & $W(G)$\\
\hline 
\begin{tikzpicture}[scale =.6]
\filldraw (0,0) circle [radius=.5mm]--(1,0) circle [radius=.5mm]--(2,0) circle [radius=.5mm]--(3,0) circle [radius=.5mm];
\filldraw (0.3,0.8) circle [radius=.5mm] (1,1) circle [radius=.5mm] (1.7,.8) circle [radius=.5mm];
\draw  (0.3,0.8)--(1,0) --(1,1) (1,0)--(1.7,.8);
\end{tikzpicture} & 40 \\
\hline
\begin{tikzpicture}[scale =.6]
\filldraw (0,0) circle [radius=.5mm]--(1,0) circle [radius=.5mm]--(2,0) circle [radius=.5mm]--(3,0) circle [radius=.5mm];
\filldraw (0.3,0.8) circle [radius=.5mm] (1,1) circle [radius=.5mm] (2,1) circle [radius=.5mm];
\draw  (0.3,0.8)--(1,0) --(1,1) (2,0)--(2,1);\end{tikzpicture}& 42\\
\hline
\begin{tikzpicture}[scale =.5]
\filldraw (0,0) circle [radius=.5mm]--(1.2,0) circle [radius=.5mm]--(1.5,1.2) circle [radius=.5mm]--(.6,2) circle [radius=.5mm]--(-0.2,1.2) circle [radius=.5mm]--(0,0) (1.7,2.7) circle [radius=.5mm] (.6,3.3)circle [radius=.5mm];
\draw (.6,3.3)--(.6,2) (1.5,1.2)--(1.7,2.7);	
\end{tikzpicture}& 40\\
\hline
\end{tabular}
\caption {The non isomorphic edge-minimal graphs in $\mathfrak{C}_{7,2}(3)$}\label{Trees in 7-2-3}
\end{center}
\label{default}
\end{table}%
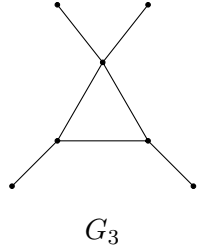
\begin{figure}[h!]
\begin{center}
\begin{tikzpicture}[scale =.6]
\filldraw (0,0) circle [radius=.5mm]--(2,0) circle [radius=.5mm]--(1,1.73) circle [radius=.5mm]--(0,0);
\filldraw (-1,-1) circle [radius=.5mm]--(0,0) (3,-1) circle [radius=.5mm]--(2,0); 
\filldraw (0,3) circle [radius=.5mm] (2,3) circle [radius=.5mm]; 
\draw (0,3)--(1,1.73)--(2,3);
\draw (1,-2) node {$G_3$};
\end{tikzpicture}
\caption{The only edge-minimal graph in $\mathfrak{C}_{7,3}(3)$}\label {edge-minimal in C-7-3-3}
\end{center}
\end{figure}

\begin{lemma}\label{block order in C-9-3-4}
Let $G_0\in \mathfrak{C}_{9,3}(4)$ be an edge-minimal graph which is not reducible to a tree in $\mathfrak{T}_9(4)$. Then $G_0$ has a non-pendant block of order $5$ or $6$.
\end{lemma}
\begin{proof}
If all the non-pendant blocks of $G_0$ are $K_2$, then $G_0$ must have exactly two $K_2$ adjacent non-pendant blocks, and it is straightforward that $G_0$ is reducible to a tree in $\mathfrak{T}_9(4)$. Hence, $G_0$ has a non-pendant block of order at least $3$. Let $B$ be such a block in $G_0$.\\

Suppose $|V(B)|=3$. Then $B$ is a triangle, say $B\cong a\sim b\sim c\sim a$. First suppose $B$ contains only $2$ cut vertices of $G_0$ say $a$ and $b$. Delete the edge  $\{bc\}$. Then $G_0\setminus \{bc\}\in \mathfrak{C}_{9,3}(4)$, which contradicts the minimality of $G_0$. Now suppose all three vertices $a,b$ and $c$ of $B$ are cut vertices. As $G_0$ has $3$ cut vertices, except $B$ all other blocks of $G_0$ are pendant. Further, since $diam(G_0)=4$, $G_0$ must have a pendant block of order $4$ or $5$. Suppose the pendant block sharing $a$ is of order $4$ or $5$. Then it can be easily seen that $G_0$ is isomorphic to one of the graphs $G_4, G_5$ or $G_6$ of Figure \ref{edge-minimal in C-9-3-4-triangle}, and so $G_0\setminus \{bc\}\in \mathfrak{C}_{9,3}(4)$, which contradicts the minimality of $G_0$.\\

Now suppose $|V(B)|=4$. Then $B$ must contain a $4$-cycle. Let $V(B)=\{a, b, c, d\}$ and $\{ab\}, \{bc\}, \{cd\}, \{da\}\in E(G_0)$. First assume that $B$ contains all the three cut vertices of $G_0$, and let these vertices be $a, b$ and $c$. Then it is not difficult to see that $G_0\setminus \{cd\}\in \mathfrak{C}_{9,3}(4)$, which contradicts the minimality of $G_0$. Now assume that $B$ contains only two cut vertices of $G_0$. These two cut vertices must be adjacent. If not, without loss of generality assume that $a$ and $c$ are the cut vertices of $G_0$ in $B$ and $d(a,c)=2$. As $diam (G_0)=4$, every vertex of $G_0\setminus B$ is adjacent to either $a$ or $c$, which implies $a$ and $c$ are the only cut vertices of $G_0$, a contradiction. So the two cut vertices must be adjacent. Now it is enough to consider two cases: $a$ and $b$ as cut-vertices, and $a$ and $c$ as cut-vertices, where $\{ac\} \in E(G_0)$.  If $a$ and $b$ are the cut vertices, then $G_0\setminus\{cd\}\in \mathfrak{C}_{9,3}(4)$, which is a contradiction. If $a$ and $c$ are the cut vertices, then also $G_0\setminus \{cd\} \in \mathfrak{C}_{9,3}(4)$, a contradiction. \\

Therefore $|V(B)|\geq 5$. Also as $G_0$ has $3$ cut vertices, $|V(B)|\leq 6$. Hence order of $B$ is either $5$ or $6$. 
\end{proof}

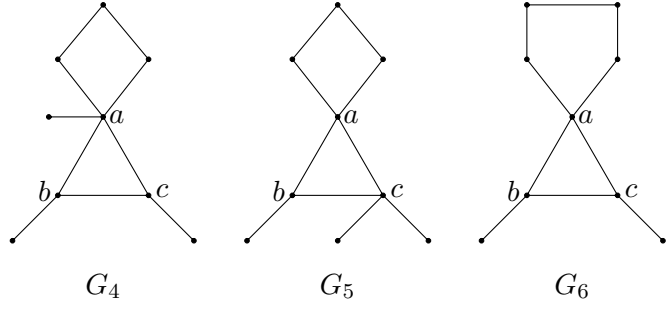
\begin{figure}[h!]
\begin{center}
\begin{tikzpicture}[scale =.6]
\filldraw (0,0) circle [radius=.5mm]--(2,0) circle [radius=.5mm]--(1,1.73) circle [radius=.5mm]--(0,0);
\filldraw (-1,-1) circle [radius=.5mm]--(0,0) (3,-1) circle [radius=.5mm]--(2,0); 
\filldraw (0,3) circle [radius=.5mm] (2,3) circle [radius=.5mm] (1,4.2) circle [radius=.5mm] (-.2,1.73) circle [radius=.5mm]--(1,1.73); 
\draw (1,4.2)--(0,3)--(1,1.73)--(2,3)--(1,4.2);
\draw (1,-2) node {$G_4$};
\draw (-.3,.1) node{$b$} (2.3,0.1) node{$c$} (1.3,1.73) node {$a$};
\end{tikzpicture}
\hskip .5cm
\begin{tikzpicture}[scale =.6]
\filldraw (0,0) circle [radius=.5mm]--(2,0) circle [radius=.5mm]--(1,1.73) circle [radius=.5mm]--(0,0);
\filldraw (-1,-1) circle [radius=.5mm]--(0,0) (3,-1) circle [radius=.5mm]--(2,0); 
\filldraw (0,3) circle [radius=.5mm] (2,3) circle [radius=.5mm] (1,4.2) circle [radius=.5mm] (1,-1) circle [radius=.5mm]--(2,0); 
\draw (1,4.2)--(0,3)--(1,1.73)--(2,3)--(1,4.2);
\draw (1,-2) node {$G_5$};
\draw (-.3,.1) node{$b$} (2.3,0.1) node{$c$} (1.3,1.73) node {$a$};
\end{tikzpicture}
\hskip .5cm
\begin{tikzpicture}[scale =.6]
\filldraw (0,0) circle [radius=.5mm]--(2,0) circle [radius=.5mm]--(1,1.73) circle [radius=.5mm]--(0,0);
\filldraw (-1,-1) circle [radius=.5mm]--(0,0) (3,-1) circle [radius=.5mm]--(2,0); 
\filldraw (0,3) circle [radius=.5mm] (2,3) circle [radius=.5mm]  (0,4.2) circle [radius=.5mm]--(2,4.2) circle [radius=.5mm]; 
\draw (0,4.2)--(0,3)--(1,1.73)--(2,3)--(2,4.2);
\draw (1,-2) node {$G_6$};
\draw (-.3,.1) node{$b$} (2.3,0.1) node{$c$} (1.3,1.73) node {$a$};
\end{tikzpicture}
\caption{The non isomorphic edge-minimal graphs in $\mathfrak{C}_{9,3}(4)$ with a triangle non-pendant block}\label {edge-minimal in C-9-3-4-triangle}
\end{center}
\end{figure}

With the help of Lemma \ref{block order in C-9-3-4}, the graphs in $\mathfrak{C}_{9,3}(4)$ that contain a non-pendant block of order $5$ or $6$ can be checked for edge-minimality. These graphs can be further filtered to identify those that are not reducible to a tree in $\mathfrak{T}_9(4)$. We have listed them in Figure \ref{edge-minimal in 9-3-4}, and using (\ref{WI in terms of distance}), their Wiener indices can be computed as follows: $W(G_7)=84$, $W(G_8)=82$, $W(G_9)=80$ and $W(G_{10})=79$.

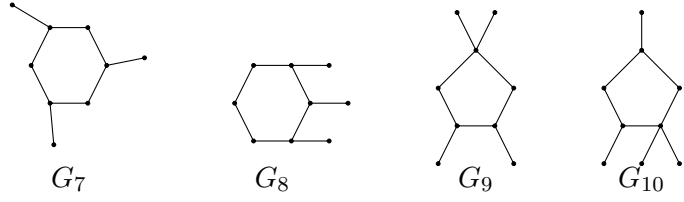
\begin{figure}[h!]
\begin{center}
\begin{tikzpicture}[scale =.5]
\filldraw (0,0) circle [radius=.5mm]--(1,0) circle [radius=.5mm]--(1.5,1) circle [radius=.5mm]--(1,2) circle [radius=.5mm]--(0,2) circle [radius=.5mm]--(-.5,1) circle [radius=.5mm]--(0,0);
\filldraw (-1,2.6) circle [radius=.5mm] -- (0,2)  (2.5,1.2) circle [radius=.5mm] -- (1.5,1) (.1,-1.1) circle [radius=.5mm] -- (0,0);
\draw (.5,-2) node {$G_7$};
\end{tikzpicture} 
\hskip 1cm
\begin{tikzpicture}[scale =.5]
\filldraw (0,0) circle [radius=.5mm]--(1,0) circle [radius=.5mm]--(1.5,1) circle [radius=.5mm]--(1,2) circle [radius=.5mm]--(0,2) circle [radius=.5mm]--(-.5,1) circle [radius=.5mm]--(0,0);
\filldraw (1,0)--(2,0) circle [radius=.5mm]  (1.5,1) -- (2.5,1) circle [radius=.5mm]  (1,2)-- (2,2) circle [radius=.5mm];
\draw (.5,-1) node{$G_8$};
\end{tikzpicture}
\hskip 1 cm
\begin{tikzpicture}[scale =.5]
\filldraw (0,0) circle [radius=.5mm]--(1,0) circle [radius=.5mm]--(1.5,1) circle [radius=.5mm]--(.5,2) circle [radius=.5mm]--(-.5,1) circle [radius=.5mm]--(0,0);
\filldraw (1,0)--(1.5,-1) circle [radius=.5mm]  (1,3) circle [radius=.5mm] --(.5,2) (.5,2) -- (0,3) circle [radius=.5mm]  (0,0)-- (-.5,-1) circle [radius=.5mm];
\draw (.5,-1.4) node{$G_9$};
\end{tikzpicture}
\hskip 1cm
\begin{tikzpicture}[scale =.5]
\filldraw (0,0) circle [radius=.5mm]--(1,0) circle [radius=.5mm]--(1.5,1) circle [radius=.5mm]--(.5,2) circle [radius=.5mm]--(-.5,1) circle [radius=.5mm]--(0,0);
\filldraw (1,0)--(.5,-1)circle [radius=.5mm] (1,0)--(1.5,-1) circle [radius=.5mm]   (.5,2) -- (.5,3) circle [radius=.5mm]  (0,0)-- (-.5,-1) circle [radius=.5mm];
\draw (.5,-1.4) node{$G_{10}$};
\end{tikzpicture}
\hskip 1cm
\caption {The non isomorphic edge-minimal graphs in $\mathfrak{C}_{9,3}(4)$ not reducible to a tree in $\mathfrak{T}_9(4)$}\label{edge-minimal in 9-3-4}
\end{center}
\end{figure}
Next we identify the edge-minimal graphs in $\mathfrak{C}_{n,n-4}(\frac{n-1}{2})$ for $n\geq 7$.
\begin{lemma}\label{4t+3 not possible}
Let $n\geq 7$ be an odd integer. 
\begin{enumerate}
\item[$(i)$] If $n\equiv 1\pmod{4}$, then $\mathfrak{T}_{n,n-4}(\frac{n-1}{2})=\{T_{21}\}$, where $T_{21}$ is the tree in Figure \ref{n cong 1 mod 4}. 
\item[$(ii)$] If $n\equiv 3 \pmod{4}$, then  $\mathfrak{T}_{n,n-4}(\frac{n-1}{2})=\emptyset$. 
\end{enumerate}
\end{lemma}
\begin{proof}
Both $(i)$ and $(ii)$ are straightforward. We present the proof only for part $(ii)$. Part $(i)$ can be verified easily or established using arguments similar to those used for $(ii)$.
\begin{enumerate}
\item[$(ii)$] Let $n\equiv 3 \pmod{4}$. Then there exists a positive integer $t$ such that $n=4t+3$.  Let $T\in \mathfrak{T}_{4t+3,4t-1}(2t+1)$ and let $P:v_0v_1\cdots v_{2t+1}$ be a longest path in $T$. As $T$ has only $4$ non-cut vertices, the only possibilities are, either exactly one $v_i$ has degree $3$ in $G$ for some $0\leq i\leq 2t+1$, or $T$ consists of $P$ and two paths, each attached to a vertex of $P$. The former possibility can be discarded as in that case $diam(T)$ exceeds $2t+1$, which can be easily observed. For the latter case, let the two paths be $P_q$ and $P_r$ such that $P$ shares only $v_i$ with $P_q$ and only $v_j$ with $P_r$.   WLOG we may assume that $q\geq r$ and $i\leq t$. Also $q+r=2t+3$, which implies $q\geq t+2$. Let $z$ be the pendant vertex of $T$ on $P_q$. Then $d_T(v_{2t+1}, z)=d_T(v_{2t+1}, v_i)+d_T(v_i, z)\geq 2t+1-i+t+1=3t+2-i\geq 2t+2$. This is a contradiction as $diam(T)=2t+1$. 
\end{enumerate}
\end{proof}

\begin{lemma}\label{edge-minimal graphs in n-4 cut vertices graphs}
Let $n\geq 7$ be an odd integer and $G$ be an edge-minimal graph in $\mathfrak{C}_{n,n-4}(\frac{n-1}{2})$. 
\begin{enumerate}
\item[$(i)$] If $n\equiv 1 \pmod{4}$, then $G\cong T_{21}$, where $T_{21}$ is the tree in Figure \ref{n cong 1 mod 4}.
\item[$(ii)$] If $n\equiv 3 \pmod{4}$, then $G\cong G_{12}$, where $G_{12}$ is the graph in Figure \ref{n cong 3 mod 4}.
\end{enumerate}
\end{lemma}
\begin{proof}
\begin{enumerate}
\item[$(i)$] Let $n \ge 7$ and $n \equiv 1 \pmod{4}$. Then there exists a positive integer $t\geq 2$ such that $n=4t+1$. Let $G$ be an edge-minimal graph in $\mathfrak{C}_{4t+1,4t-3}(2t)$.  Suppose $G$ is not a tree. Since $G$ contains $4$ non-cut vertices, it has a block $B$ such that $|V(B)|\in \{3,4\}$. Further, as $G$ is edge-minimal in $\mathfrak{C}_{4t+1,4t-3}(2t)$, $B$ can not be a pendant block. So $B$ is a non-pendant block. \\

First suppose $|V(B)|=3$ and let $B\cong v_q\sim v_r\sim v_s\sim v_q$. $G\setminus E(B)$ has three connected components, say $G_q$, $G_r$ and $G_s$ containing $v_q, v_r$ and $v_s$, respectively. As $G$ has exactly $4$ non-cut vertices,  two of these components will be a path with a pendant vertex in $\{v_q, v_r, v_s\}$. WLOG assume that $G_r\cong P_r$ and $G_s\cong P_{s}$ with $r\geq s\geq 1$ such that $v_r$ and $v_s$ are pendant vertices in $P_r$ and $P_s$, respectively.  Let $z$ be a vertex in $G_q$ farthest from $v_q$ and say $d(v_q,z)=q$. As $diam(G)=2t$, we have $r+s\leq 2t+1$. If $r+s<2t+1$, then delete the edge $\{v_rv_s\}$. Deleting $\{v_rv_s\}$ has no effect on the number of cut vertices. Further as $r+s< 2t+1$, $G_0\setminus \{v_rv_s\}$ has the same diameter as $G$. So $G\setminus \{v_rv_s\}\in \mathfrak{C}_{4t+1,4t-3}(2t)$, which contradicts the edge-minimality of $G$ in $\mathfrak{C}_{4t+1,4t-3}(2t)$. Therefore $r+s=2t+1$. This implies $r\geq t+1$ and $s\leq t$. As $diam(G)=2t$,  
\begin{align*}
q+r\leq 2t\implies q\leq t-1\implies s+q\leq 2t-1.
\end{align*}
Therefore $G\setminus \{v_qv_s\}\in \mathfrak{C}_{4t+1,4t-3}(2t)$, which contradicts the edge-minimality of $G$ in $\mathfrak{C}_{4t+1,4t-3}(2t)$.\\

Now suppose $|V(B)|=4$. Then $B$ contains a $4$-cycle say $v_a\sim v_b\sim v_c\sim v_d\sim v_a$. As $G$ has $4$ non-cut vertices, $G\setminus E(B)$ is a disjoint union of  $4$ paths say $P_a, P_b, P_c$ and $P_d$ containing $v_a, v_b, v_c$ and $v_d$, respectively. We may assume that $a\geq b\geq c\geq d\geq 1$. As $a+b+c+d=4t+1$, $a+b\geq 2t+1$. Also since $diam(G)=2t$, $a+b\leq 2t+1$. So $a+b=2t+1$. Therefore $c+d=2t$. As $a\geq b$, $a\geq t+1$. This implies $b\leq t$, and hence $c\leq t$ and $d\leq t$. As $c+d=2t$, $c=d=t$. Also $a+c\leq 2t+1$, which implies $a\leq t+1$. Therefore  $a=t+1$ and $b=c=d=t$. This implies $G\cong G_{11}$, where $G_{11}$ is the graph shown in Figure \ref{G11}. But $G_{11}\setminus \{\{v_bv_c\}, \{v_cv_d\}\}\in \mathfrak{C}_{4t+1,4t-3}(2t)$, which contradicts the minimality of $G$.

Therefore $G$ must be a tree, and hence by Lemma \ref{4t+3 not possible} $(i)$, $G\cong T_{21}$.
\begin{figure}[h!]
\begin{center}
\begin{tikzpicture}
\filldraw (0,0) node[below][scale=.8]{$v_d$} circle [radius=.5mm]--(1,0) node[below][scale=.8]{$v_c$} circle [radius=.5mm]--(1,1) circle [radius=.5mm]--(0,1)  circle [radius=.5mm];
\draw (0,0)--(0,1) (1,0)--(1.5,-.25) (0,0)--(-.5,-.25) (0,1)--(-.5,1.25) (1,1)--(1.5,1.25) (0,1)--(1,0);
\filldraw (1.5,-.25)circle [radius=.5mm][dash pattern= on 1pt off 1pt]--(3,-1)circle [radius=.5mm]  (-.5,-.25)circle [radius=.5mm]--(-2,-1)circle [radius=.5mm] (-.5,1.25)circle [radius=.5mm]--(-2.5,2.25)circle [radius=.5mm] (1.5,1.25)circle [radius=.5mm]--(3,2)circle [radius=.5mm] ;
\draw (.95,.94) node[above][scale=.8]{$v_b$} (0,.94) node[above][scale=.8]{$v_a$}; 
\draw[decoration={brace,raise=5pt},decorate]
  (1,0) -- (3,-1);
 \draw (2.3,0)node {$P_{t}$};
 \draw[decoration={brace,mirror,raise=5pt},decorate]
  (0,0) -- (-2,-1);
 \draw (-1,0)node {$P_{t}$};
 \draw[decoration={brace,raise=5pt},decorate]
  (1,1) -- (3,2);
 \draw (2,2)node {$P_{t}$};
 \draw[decoration={brace,mirror,raise=5pt},decorate] (0,1) -- (-2.5,2.25);
 \draw (-1,2.1)node {$P_{t+1}$};
 \draw (.5,-1) node{$G_{11}$};
\end{tikzpicture}
\caption {The only edge-minimal graph in $\mathfrak{C}_{n,n-4}(\frac{n-1}{2})$ containing a non-pendant block of order $4$ when $n=4t+1$}\label{G11}
\end{center}
\end{figure}
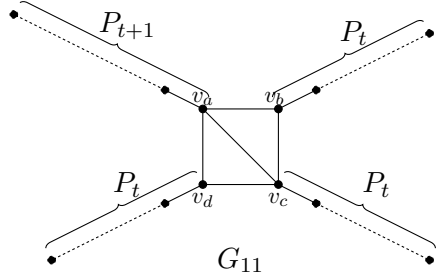

\item[$(ii)$] Let $n\geq 7$ and $n\equiv 3 \pmod{4}$. Then there exists a positive integer $t$ such that $n=4t+3$. Let $G$ be an edge-minimal graph in $\mathfrak{C}_{4t+3,4t-1}(2t+1)$. By Lemma \ref{4t+3 not possible} $(ii)$, $G$ can not be a tree. As $G$ is edge-minimal in $\mathfrak{C}_{4t+3,4t-1}(2t+1)$, each pendant block of $G$ must be $K_2$. Let $B$ be a non-pendant block of $G$ different from $K_2$. Then $|V(B)|\in \{3,4\}$. If $|V(B)|=3$, then it is straightforward that $G$ is isomorphic to the graph $G_{12}$ in Figure \ref{n cong 3 mod 4}. Suppose $|V(B)|=4$. Then $B$ contains a $4$ cycle say $v_av_bv_cv_d$ where $\{v_av_b\}, \{v_bv_c\}, \{v_cv_d\}$ and $\{v_dv_a\}\in E(B)$. As $G$ has $4$ non-cut vertices, $G\setminus E(G)$ has four paths. Let these paths be $P_a$, $P_b$, $P_c$ and $P_d$ containing $v_a, v_b, v_c$ and $v_d$, respectively. Let $a\geq b\geq c\geq d$. Then it can be easily concluded that $a=b=c=t+1,d=t$ and $\{v_av_c\}\in E(G)$. Now $G\setminus \{v_av_d\}\in \mathfrak{C}_{4t+3, 4t-1}(2t+1)$, which is a contradiction. So $|V(B)|$ can not be $4$. This completes the proof.

 \end{enumerate}
\end{proof}

\begin{figure}[h!]
    \centering
    \begin{subfigure}[b]{0.45\textwidth}
        \centering
   \begin{tikzpicture}
\filldraw (0,0) node[above][scale=.8]{$v_0$} circle [radius=.5mm]--(1,0) node[above][scale=.8]{$v_1$} circle [radius=.5mm](3,0) node[above][scale=.8]{$v_t$} circle [radius=.5mm] (5,0) node[above][scale=.8]{$v_{2t-1}$} circle [radius=.5mm]--(6,0) node[above][scale=.8]{$v_{2t}$} circle [radius=.5mm];
\filldraw (2.5,-.8) circle [radius=.5mm] (1.5,-2.4) circle [radius=.5mm] (3.5,-.8) circle [radius=.5mm] (4.5,-2.4) circle [radius=.5mm];
\draw (1,0) [dash pattern= on 1pt off 1pt]--(5,0);
\draw (1.5,-2.4)[dash pattern= on 1pt off 1pt]--(2.5,-.8) (3.5,-.8)[dash pattern= on 1pt off 1pt]--(4.5,-2.4);
\draw (2.5,-.8)--(3,0)--(3.5,-.8);
\draw (3,-3) node{$T_{21}$};
 \draw[decoration={brace,raise=5pt},decorate] (3,0) -- (4.5,-2.4);
  \draw[decoration={brace,mirror,raise=5pt},decorate] (3,0) -- (1.5,-2.4);
  \draw (4.5,-1.2) node[scale =.8]{$P_{t+1}$} (1.5,-1.2) node[scale =.8]{$P_{t+1}$};
\end{tikzpicture}
        \caption{when $n=4t+1$}
        \label{n cong 1 mod 4}
    \end{subfigure}
    \hfill 
    \begin{subfigure}[b]{0.45\textwidth}
        \centering
\begin{tikzpicture}
\filldraw (0,0) node[below][scale=.5]{$B$} circle [radius=.5mm]--(1,0) node[below][scale=.5]{$C$} circle [radius=.5mm]--(.5,.8) node[above][scale=.5]{$A$} circle [radius=.5mm];
\draw (.5,.8)--(0,0)--(-.8,-.4) (1,0)--(1.8,-.4) (1.3,1.35)--(.5,.8)--(-.3,1.35);
\filldraw (-.8,-.4) circle [radius=.5mm] (-2.4,-1.2) circle [radius=.5mm] (1.8,-.4) circle [radius=.5mm] (3.4,-1.2) circle [radius=.5mm];
\draw (-.8,-.4)[dash pattern=on 1pt off 1pt]--(-2.4,-1.2) (1.8,-.4) --(3.4,-1.2);
\filldraw (-2,2.5) circle [radius=.5mm](-.3,1.35) circle [radius=.5mm] (1.3,1.35) circle [radius=.5mm] (3,2.5) circle [radius=.5mm];
\draw (-2,2.5) [dash pattern=on 1pt off 1pt]--(-.3,1.35) (1.3,1.35)--(3,2.5);
 \draw [decoration={brace,raise=5pt},decorate] (0.5,.8) -- (3,2.5);
\draw  [decoration={brace,mirror,raise=5pt},decorate](.5,.8)--(-2,2.5);
\draw  [decoration={brace,mirror,raise=5pt},decorate](0,0)--(-2.4,-1.2);
\draw  [decoration={brace,raise=5pt},decorate](1,0)--(3.4,-1.2);
\draw (2.5,0) node {$P_{t+1}$} (-1.5,0) node {$P_{t+1}$} (1.5,2.2) node {$P_{t+1}$} (-.5,2.2) node {$P_{t+1}$} (0.5,-1) node {$G_{12}$};
\end{tikzpicture}
        \caption{when $n=4t+3$}
        \label{n cong 3 mod 4}
    \end{subfigure}

    \caption{The only edge-minimal graphs in $\mathfrak{C}_{n,n-4}(\frac{n-1}{2})$}\label{}
\end{figure}
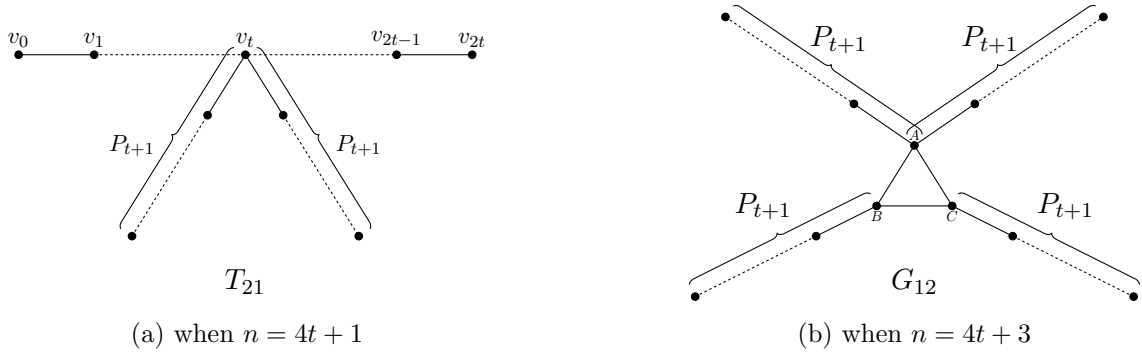

\section{Graphs with a fixed diameter}\label{Graphs with cut vertices}

Let $n\geq 7$ be an odd integer and $G \in \mathfrak{C}_n(\frac{n-1}{2})$. Then $G$ has some (may be zero) cut vertices. As $P_n$ is the only graph with $n-2$ cut vertices and $diam(P_n)=n-1$, $G$ can not have $n-2$ cut vertices. Further, with some analysis it can be observed that $G$ also can not have $n-3$ cut vertices. This implies $G\in \mathfrak{C}_{n,k}(\frac{n-1}{2})$ for some $k\in \{0,1,\ldots, n-4\}$. One approach to prove Conjecture \ref{the conjecture} could be to prove it for all $G\in \mathfrak{C}_{n,k}(\frac{n-1}{2})$ and $k\in \{0,1,\ldots, n-4\}$. We prove that the conjecture is true for $G \in \mathfrak{C}_{n,k}(\frac{n-1}{2})$ when $k=0,1,2, 3$ or $n-4$. The graphs with the maximum Wiener index in $\mathfrak{C}_{n,0}$ have been characterised in \cite{Plesnik}.

\begin{proposition}[\cite{Plesnik}, Theorem 5]\label{2-connected max}
Let $G\in \mathfrak{C}_{n,0}$, $n\geq 3$. Then $W(G)\leq W(C_n)$ and the equality holds if and only if $G\cong C_n$. 
\end{proposition}

\noindent The graphs with the maximum Wiener index in $\mathfrak{C}_{n,k}$ for $1\leq k\leq 3$ have been characterised in \cite{Pandey 1}.
\begin{proposition}[\cite{Pandey 1}, Theorem5.5, Theorem 5.7, Theorem 5.9]\label{max fixed cut vertices}
Let $1\leq k\leq 3$ and $G\in \mathfrak{C}_{n,k}$.
\begin{enumerate}
\item[$(i)$] If $k=1$ and $n\geq 7$, then $W(G)\leq W(L_{n,n-1})$ and the equality holds iff $G\cong L_{n,n-1}$,
\item[$(ii)$] if k=2 and $n\geq 10$, then $W(G)\leq W(L_{n,n-2})$ and the equality holds iff $G\cong L_{n,n-2}$, and 
\item[$(iii)$] if $k=3$ and $n\geq 14$, then $W(G)\leq W(L_{n,n-3})$ and the equality holds iff $G\cong L_{n,n-3}$.
\end{enumerate}
\end{proposition}

\noindent We also note the following two remarks from \cite{Pandey 1}.
\begin{remark}[\cite{Pandey 1}, Table 2]\label{Rmk L-9-2}
Let $G\in \mathfrak{C}_{9,2}$. Then $W(G)\leq 88$ and the equality is attained only by $L_{9,7}$ and the graph obtained by identifying a pendant vertex of $K_{1,3}$ with a vertex of $C_6$.
\end{remark}

\begin{remark}[\cite{Pandey 1}, Table 3]\label{Rmk L-13-3}
Let $G\in \mathfrak{C}_{13,3}$. Then $W(G)\leq 264$ and the equality is attained only by $L_{13,10}$ and the graph obtained by identifying the pendant vertex of $L_{8,6}$ with a vertex of $C_6$.
\end{remark}

\noindent The problem of characterising the graphs with maximum Wiener index in $\mathfrak{C}_{n,k}$, $k\geq 4$ is left open in \cite{Pandey 1}.  An {\it $s$- pendant block} (see \cite{Pandey 1}) of a graph is a pendant block which shares its cut vertex with exactly one non-pendant block. 
\begin{lemma}[\cite{Pandey 1}, Proposition 3.3]\label{s pendant blocks}
Let $G\in \mathfrak{C}_{n,k}$, $k\geq 2$. Then $G$ has at least two vertex-disjoint $s$-pendant blocks.
\end{lemma}

\noindent Recall that the graph $L_{n, n-k} \in \mathfrak{C}_{n,k}$ and it cotains exactly one pendant vertex. Further $diam (L_{n,n-k})=k+\lfloor\frac{n-k}{2}\rfloor$.

\begin{lemma}[\cite{Pandey 1}, Theorem 4.8] \label{Dmax}
Let $z$ be the pendant vertex of $L_{n,n-k}$ and $G\in \mathfrak{C}_{n,k}$,  $1\leq k\leq n-3$. Then $D_G(v)\leq D_{L_{n,n-k}}(z)$ for any $v\in V(G)$. 
\end{lemma}
\noindent Table  \ref{MaxWI_2} lists the graphs attaining maximum Wiener in $\mathfrak{C}_{n,2}$, $6\leq n\leq 10$. The entries in the table for $6\leq n\leq 9$ have been adapted from \cite{Pandey 1} (see Table 2 in  \cite{Pandey 1}). For $n=10$, the entries follow from Proposition \ref{max fixed cut vertices} $(ii)$. Now we prove that the Conjecture \ref{the conjecture} is true for graphs with at most $3$ cut vertices.
\begin{table}[h!]
\begin{center}
\begin{tabular}{|c|c|c|}
\hline
$n$&Graphs $G$ having maximum W.I. over $\mathfrak{C}_{n,2}$&$W(G)$\\
\hline
$6$& $L_{6,4}$ and \hskip .3cm \begin{tikzpicture}[scale=.5]
\filldraw (-0.9,0.6,0) circle [radius=.5mm]--(0,0) circle [radius=.5mm]--(1,0) circle [radius=.5mm]--(1.9,0.6) circle [radius=.5mm];
\filldraw (-0.9,-0.6) circle [radius=.5mm]--(0,0);
\filldraw (1,0)--(1.9,-0.6) circle [radius=.5mm];
\end{tikzpicture}
& $29$\\
\hline
$7$&\hskip 2cm \begin{tikzpicture}[scale=.5]
\filldraw (1,0) circle [radius=.5mm]--(2,0) circle [radius=.5mm]--(2.9,0.6) circle [radius=.5mm]--(3.8,0) circle [radius=.5mm]--(2.9,-0.6) circle [radius=.5mm];
\draw (2,0)--(2.9,-0.6); 
\filldraw (0.1,0.6)circle [radius=.5mm]--(1,0);
\filldraw(1,0)--(0.1,-0.6) circle [radius=.5mm];
\end{tikzpicture}
& $44$\\
\hline
$8$ &  $L_{8,6}$ and
\begin{tikzpicture}[scale=.5]
\filldraw (0,0) circle [radius=.5mm]--(.5,.5) circle [radius=.5mm]--(1,0) circle [radius=.5mm]--(.5,-.5) circle [radius=.5mm]   (2,0) circle [radius=.5mm]--(2.5,.5) circle [radius=.5mm]--(3,0) circle [radius=.5mm]--(2.5,-.5) circle [radius=.5mm];
\draw(0,0)--(.5,-.5) (1,0)--(2,0)--(2.5,-.5);
\end{tikzpicture}
& $64$\\
\hline
$9$& $L_{9,7}$ and \begin{tikzpicture}[scale=.5]
\filldraw (0.1,0.5) circle [radius=.5mm]--(1,0) circle [radius=.5mm]--(2,0) circle [radius=.5mm]--(2.9,0.6) circle [radius=.5mm]--(3.8,0.6) circle [radius=.5mm]--(4.7,0) circle [radius=.5mm]--(3.8,-0.6) circle [radius=.5mm]--(2.9,-0.6) circle [radius=.5mm];
\draw (2,0)--(2.9,-0.6); 
\draw (0.1,-0.5) circle [radius=.5mm]--(1,0);
\end{tikzpicture}
& $88$ \\
\hline
$10$ & $L_{10,8}$ & $121$\\
\hline
\end{tabular}
\end{center}
\label{default}
\caption{ The graphs having maximum Wiener index over $\mathfrak{C}_{n,2}$ for $6\leq n\leq 10$} \label{MaxWI_2}
\end{table}

\begin{theorem}
Let $n\geq 7$ be an odd integer and $G\in \mathfrak{C}_{n,k}(\frac{n-1}{2})$, $0\leq k \leq 3$. Then $W(G)\leq W(C_n)$.
\end{theorem}
\begin{proof}
Suppose $G\in \mathfrak{C}_{n,0}(\frac{n-1}{2})$. It follows from proposition \ref{2-connected max} that $W(G)\leq W(C_n)$.\\

Suppose $G\in \mathfrak{C}_{n,1}(\frac{n-1}{2})$. Then $G\ncong L_{n,n-1}$ because $diam(L_{n,n-1})=\frac{n+1}{2}$. Therefore by  Proposition \ref{max fixed cut vertices} $(i)$,
\begin{align*}
W(G)&<W(L_{n,n-1}) \\
	&=\frac{n^3-n^2+7n-7}{8} &\mbox{[using $g=n-1$ in (\ref{WI_Lnk})]}\\ 
	&\leq \frac {n^3-n}{8}\\
	&=W(C_n).
\end{align*}

Suppose $G\in \mathfrak{C}_{n,2}(\frac{n-1}{2})$. Clearly $G\ncong L_{n,n-2}$ as $diam(L_{n,n-2})=\frac{n+1}{2}$. If $n\geq 11$, then by Proposition \ref{max fixed cut vertices} $(ii)$,
\begin{align*}
W(G)&<W(L_{n,n-2})\\ 
	&=\frac{n^3-2n^2+19n-34}{8} &\mbox{[using $g=n-2$ in (\ref{WI_Lnk})]}\\ 
	&<\frac {n^3-n}{8}\\
	&=W(C_n).
\end{align*}
Further, the non-isomorphic edge-minimal graphs in $\mathfrak{C}_{7,2}(3)$ are the three graphs shown in Table \ref{Trees in 7-2-3} and the maximum Wiener index among them is $42$. So, if $G\in \mathfrak{C}_{7,2}(3)$, then $W(G)\leq 42=W(C_7)$. If $G\in \mathfrak{C}_{9,2}(3)$, then by Remark \ref{Rmk L-9-2}, $W(G)< 88< W(C_9)$.

Finally suppose $G\in \mathfrak{C}_{n,3}( \frac{n-1}{2})$. Clearly $G\ncong L_{n,n-3}$. If $n\geq 15$, then by Proposition \ref{max fixed cut vertices} $(iii)$,
\begin{align*}
W(G)&< W(L_{n,n-3})\\
        &= \frac{n^3-3n^2+39n-85}{8} &\mbox{[using $g=n-3$ in (\ref{WI_Lnk})]}\\ 
        &< \frac{n^3-n}{8} \\
        &=W(C_n).
\end{align*}
 Let $G\in \mathfrak{C}_{7,3}(3)$. The graph $G_3$ in Figure \ref{edge-minimal in C-7-3-3} is the only edge-minimal graph in $\mathfrak{C}_{7,3}(3)$ and $W(G_3)=40$. Therefore, $W(G)\leq 40< W(C_7)$. \\
 
 Let $G\in \mathfrak{C}_{9,3}(4)$. If $G$ is a tree or reducible to a tree in $\mathfrak{T}_9(4)$, then by Proposition \ref{Wagner}, $W(G)\leq W(S(2,3,3))=90=W(C_9)$. Suppose $G$ is not reducible to a tree in $\mathfrak{T}_9(4)$. Then $G$ is isomorphic to one of the graphs $G_7, G_8, G_9$ or $G_{10}$ of Figure \ref{edge-minimal in 9-3-4}.  So $W(G)\leq \max \{W(G_7), W(G_8), W(G_9), W(G_{10})\}=84<W(C_9)$.\\
 
Now let $G\in \mathfrak{C}_{11,3}(5)$ and  let $G_0\in \mathfrak{C}_{11,3}(5)$ such that $W(G_0)=\max \{W(G): G\in \mathfrak{C}_{11,3}(5)\}$. So $W(G)\leq W(G_0)$. Clearly $G_0$ is edge-minimal in $\mathfrak{C}_{11,3}(5)$. By Lemma \ref{s pendant blocks}, $G_0$ has at least two vertex-disjoint $s$-pendant blocks. Let $B_1$ and $B_2$ be two such blocks, and let $|V(B_1)|\leq |V(B_2)|$.\\
 \noindent {\bf Claim:} $B_1$ and all the pendant blocks adjacent to $B_1$ are $K_2$. \\
 Suppose not. WLOG, assume that $B_1$ is a block with the smallest order among $B_1$ and the pendant blocks adjacent to it. Suppose $B_1$ is a triangle, say $B_1\cong a\sim b\sim c\sim a$. Let $a$ be the cut vertex of $G_0$ in $B_1$. Then the graph $G_0\setminus \{bc\}\in \mathfrak{C}_{11,3}(5)$, and $W(G_0\setminus \{bc\})>W(G_0)$, which is a contradiction. So $|V(B_1)|\geq 4$. Now let $w_1$ and $w_2$ be the cut vertices of $G_0$ in $B_1$ and $B_2$, respectively. As $G_0$ is edge-minimal in $\mathfrak{C}_{11,3}(5)$ and $B_1, B_2$ are vertex-disjoint, it follows that $d(w_1,w_2)\geq 2$. Let $w_1'\in V(B_1)$ be a vertex farthest from $w_1$ and $w_2'\in V(B_2)$ be a vertex farthest from $w_2$. If $d(w_1,w_1')=1$, then there exists a triangle $w_1\sim w_1'\sim w_1''\sim w_1$ in $B$, and we get a contradiction as $G_0\setminus \{w_1'w_1''\}\in \mathfrak{C}_{11,3}(5)$. So $d(w_1,w_1')\geq 2$. Similarly $d(w_2,w_2')\geq 2$. Therefore we get $d(w_1',w_2')\geq 6$, a contradiction. This proves the claim.\\
 
  As $diam(G_0)=5$, there can be at most $4$ $K_2$ pendant blocks adjacent to $B_1$.  Let $G_{01}$ be the subgraph of $G_0$ containing $B_1$ and all the pendant blocks adjacent to it, and $G_{02}$ be the maximal subgraph of $G_0$ containing $w_1$ as a non-cut vertex. Let $|V(G_{01})|=n_1$ and $|V(G_{02})|=n_2$. Then $G_{01}\cong K_{1,n_1-1}$ and  $G_{02}\in \mathfrak{C}_{n_2,2}$. Therefore $G_0\cong K_{1,n_1-1}\cup G_{02}$ with $V(K_{1,n_1-1})\cap V(G_{02})=\{w_1\}$. Also, $2\leq n_1\leq 6$ and $6\leq n_2 \leq 10$ with $n_1+n_2-1=11$.  Let $G_{02}'$ be a graph attaining the maximum Wiener index in $\mathfrak{C}_{n_2,2}$ and $z$ be the pendant vertex of $L_{n_2,n_2-2}$. Then
  \begin{align*}
  W(G_0)&=W(K_{1,n_1-1})+W(G_{02})+(n_2-1)D_{K_{1,n_1-1}}(w_1)+(n_1-1)D_{G_{02}}(w_1)\\
  	      &\leq (n_1-1)^2+W(G_{02}')+(n_2-1)(n_1-1)+(n_1-1)D_{L_{n_2,n_2-2}}(z). \numberthis \label{WI n1-n2}
  \end{align*}
Depending on the values of $n_1$ and $n_2$, and using the information from Table \ref{MaxWI_2}, an upper bound on the Wiener indices of all possibilities of $G_0$ have been computed in Table \ref{Max G0 n1-n2}. From Table \ref{Max G0 n1-n2}, we get that $W(G_0)\leq 164<165=W(C_{11})$. \\
 
 \begin{table}[h!]
\begin{center}
\begin{tabular}{|c|c|c|}
\hline
$n_1$& $n_2$ & $W(G_0)\leq$\\
\hline 
2& 10& 164 \\
\hline 
3& 9& 162\\
\hline 
4& 8& 160\\
\hline 
5& 7& 152\\
\hline 
6& 6& 144\\
\hline
\end{tabular}
\caption {The maximum Wiener index of $G_0$ in (\ref{WI n1-n2}) depending on the values of $n_1$ and $n_2$}\label{Max G0 n1-n2}
\end{center}
\label{default}
\end{table}%
 
Finally, let $G\in \mathfrak{C}_{13,3}(5)$. Then by Remark \ref{Rmk L-13-3}, $W(G)\leq 264< W(C_{13})$. This completes the proof.
 \end{proof}
 
\begin{theorem}
Let $n\geq 7$ be an odd integer and $G\in \mathfrak{C}_{n,n-4}(\frac{n-1}{2})$. Then $W(G)\leq W(C_n)$.
\end{theorem}
\begin{proof}
Let $G\in \mathfrak{C}_{n,n-4}(\frac{n-1}{2})$. If $n\equiv 1 \pmod{4}$, then  by Lemma \ref{edge-minimal graphs in n-4 cut vertices graphs} $(i)$, $G\cong T_{21}$ (the tree in Figure \ref{n cong 1 mod 4}), and we get 
\begin{align*}
W(G)=W(T_{21})&= \frac{5n^3+9n^2-17n+3}{48}\\
			 &<\frac{n^3-n}{8} & \mbox{[for $n\geq 9$]}\\
			 &=W(C_n).
\end{align*}
If $n\equiv 3 \pmod{4}$, then by Lemma \ref{edge-minimal graphs in n-4 cut vertices graphs} $(ii)$, $G\cong G_{12}$ (the graph in Figure \ref{n cong 3 mod 4}), and we get 
\begin {align*}
W(G)=W(G_{12})&=\frac{5n^3+6n^2-11n-12}{48}\\
                           &<\frac{n^3-n}{8} & \mbox{[for $n\geq 7$]}\\
                           &=W(C_n).
\end{align*}
Here the Wiener indices of $T_{21}$ and $G_{12}$ are calculated using Lemma \ref{count}.
\end{proof}


\section{Conclusions}\label{Conclusions}
Our aim was to prove Conjecture \ref{the conjecture}. We first studied the extremal problem of finding the maximum Wiener index of trees with a fixed diameter. We obtained a necessary condition on a tree to be of maximal Wiener index in $\mathfrak{T}_n(d)$ and using it characterised the maximal trees in $\mathfrak{T}_n(n-k)$ for $k=4$ and $5$. Using our technique, the manual analysis for other larger values of $k$ becomes difficult because of many detailed calculations and pairwise comparisons. However we believe that using computational tools this characterisation can be done for larger values of $k$. If it could be done for $\mathfrak{T}_n(\frac{n-1}{2})$, $n$ odd, then it will verify Conjecture \ref{the conjecture} when $G$ is a tree. Next we proved Conjecture \ref{the conjecture} for graphs containing $0,1,2,3$ or $n-4$ cut vertices. To prove our results for $0,1,2$ and $3$ cut vertices, we used the characterisation of graphs with maximum Wiener indices in $\mathfrak{C}_{n,k}$ for $0\leq k\leq 3$ from \cite{Pandey 1}. The characterisation of graphs with maximum Wiener index in $\mathfrak{C}_{n,k}$ for $k\geq 4$ is open. Characterising these graphs may lead to the conjecture being completely settled. \\

\noindent{\bf Acknowledgements:} The authors gratefully acknowledge Stijn Cambie for bringing the conjecture to their attention during CanaDam 2025.\\

\noindent{\bf Funding:} This work was funded by the Einwechter Centre for Supply Chain Management (a research centre funded by Mr. Dan Einwechter) at Wilfrid Laurier University.\\

\noindent{\bf Declaration of generative AI use:} ChatGPT was used to improve the wording of some sentences in this paper.

\end{document}